\RequirePackage{fix-cm}
\documentclass[envcountsame,smallextended]{svjour3}       
\smartqed  
\usepackage{amsthm}
\usepackage{amssymb}
\usepackage{graphicx}
\usepackage[top=35mm, bottom=25mm, left=20mm, right=20mm]{geometry}
\usepackage[numbers,sort&compress]{natbib}

\usepackage{amsmath}
\usepackage{amsfonts}
\usepackage{amssymb}
\usepackage{enumitem}
\usepackage{graphicx}
\usepackage{subcaption}
\usepackage{tikz}
\usetikzlibrary{arrows.meta, positioning}
\usepackage{pgfplots}
\usetikzlibrary{patterns}
\usepackage{multicol}
\usepackage{hyperref}
\usepackage{xcolor}
\usepackage{cancel}
\usepackage{ulem}
\usepackage[active]{srcltx}

\numberwithin{equation}{section}
\usepackage{subcaption}

\usepackage{booktabs}
\usepackage{multirow}

\usepackage{algorithm}%
\usepackage{algorithmicx}%
\usepackage{algpseudocode}%

\newcommand\R{\mathbb{R}}
\newcommand\Rinf{\overline{\mathbb{R}}}
\newcommand\inter[1]{ {\rm \textbf{int}}(#1)} 
\newcommand\closure[1]{ {\rm \textbf{cl}}(#1)} 
\newcommand\dom[1]{ \bs{{\rm dom}}(#1)} 
\newcommand\Dom[1]{ \bs{{\rm Dom}}(#1)} 
\newcommand\gf{\varphi} 
\newcommand\gh{\psi} 
\newcommand\fgam[3]{#1_{#3}^{#2}}

\newcommand\hfgam[3]{#1^{#3, #2}}
\newcommand\prox[3]{ \bs{{\rm prox}}_{#2#1}^{#3}}

\newcommand\ov[1]{\overline{#1}}
\newcommand\mb{\mathbf{B}}
\newcommand\bs[1]{\boldsymbol{#1}}
\newcommand\argmint[1]{\mathop{\bs{\arg\min}}\limits_{#1}}
\newcommand\Nz{\mathbb{N}_0}

\newcommand\Tprox[3]{ \bs T_{#2#1}^{#3}}

\newcommand\argmin[1]{\bs{\arg\min}_{#1}}

\journalname{}

\begin{document}

\title{On fundamental properties of high-order forward-backward envelope}

\titlerunning{On fundamental properties of high-order forward-backward envelope}        

\author{Alireza Kabgani         \and
        Masoud Ahookhosh 
}


\institute{A. Kabgani, M. Ahookhosh \at
              Department of Mathematics, University of Antwerp, Antwerp, Belgium. \\
              \email{alireza.kabgani@uantwerp.be, masoud.ahookhosh@uantwerp.be}            
}

\date{Received: date / Accepted: date}

\maketitle

\begin{abstract}
This paper studies the fundamental properties of the high-order forward–backward splitting mapping (HiFBS) and its associated high-order forward–backward envelope (HiFBE) through the lens of high-order regularization for nonconvex composite functions. Specifically, we (i) establish the boundedness and uniform boundedness of HiFBS, along with the Hölder and Lipschitz continuity of HiFBE; (ii) derive an explicit form for the subdifferentials of HiFBE; and (iii) investigate necessary and sufficient conditions for the differentiability and weak smoothness of HiFBE under suitable assumptions. By leveraging the prox-regularity of $g$ and the concept of $p$-calmness, we further demonstrate the local single-valuedness and continuity of HiFBS, which in turn guarantee the differentiability of HiFBE in neighborhoods of calm points. This paves the way for the development of gradient-based algorithms tailored to nonconvex composite optimization problems.


\end{abstract}

\section{Introduction}
\label{intro}


Over the past few decades, {\it smoothing methods} have become a central theme in nonsmooth and nonconvex optimization, forming a framework for designing gradient-based iterative methods. 
The underlying philosophy is simple yet powerful: a nonsmooth objective is replaced by a carefully designed smooth approximation that retains the main geometric and variational features of the original problem. 
This transformation enables the application of efficient first-order and generalized second-order methods, which often exhibit superior convergence behavior compared to classical subgradient or bundle-type schemes \cite{Beck12,nesterov2005smooth}. 
Accordingly, a wide variety of smoothing mechanisms have been proposed, ranging from {\it penalty} and {\it barrier smoothing} to {\it Nesterov’s dual smoothing} \cite{nesterov2005smooth} and convolution-based constructions such as {\it infimal} and {\it integral convolution} \cite{Bauschke17,Beck12,Burke2017,Burke2013Gradient,Chen2012}. 
These ideas have proved remarkably versatile, leading to numerous algorithmic developments in composite and structured optimization problems; see, e.g.,~\cite{Beck12,ben2006smoothing,Bot15,Bot2020,Ahookhosh21,patrinos2013proximal,Stella17,Themelis18,themelis2019acceleration,themelis2020douglas}. 

Among all the smoothing paradigms, the {\it Moreau envelope} \cite{Moreau65} occupies a distinguished position due to its simplicity and favorable properties.
Let $\gf: \R^n\to\Rinf:=\R\cup \{+\infty\}$ be a proper and lower semicontinuous (lsc) function. Then the {\it Moreau envelope} \cite{Moreau65} is given by
\begin{equation}\label{eq:mor}
    \fgam{\gf}{}{\gamma}(x):=\mathop{\bs{\inf}}\limits_{y\in \R^n} \left\{\gf(y)+\frac{1}{2\gamma}\Vert x- y\Vert^2\right\},
\end{equation}
for $\gamma>0$. Under suitable assumptions, the Moreau envelope provides a smooth reformulation of $\gf$ via proximal regularization, preserving the set of minimizers while improving analytical tractability. 
Owing to these nice properties, the Moreau envelope and its associated {\it proximal-point operator} have become the foundation for a broad spectrum of modern algorithms \cite{Beck12,Bot15,Burke13,Drusvyatskiy19,Ghaderi24,Kabgani25itsdeal,Kabgani24itsopt,KecisThibault15,Parikh14,Poliquin96,Rockafellar09,Shefi16}.
In recent works, high-order proximal operators and their corresponding envelopes, which generalize the classical Moreau regularization to powers $p>1$, have been revisited \cite{Ahookhosh2025,Ahookhosh2025Minimizing,Kabgani24itsopt,Kabgani25itsdeal,KecisThibault15}; see Definition~\ref{def:home}. The parameter $p$ introduces an additional degree of freedom that controls the strength and geometry of the smoothing techniques \cite{Kabganidiff}, and it has proven useful for achieving superlinear convergence rates in the class of uniformly quasiconvex functions, including strongly and uniformly convex functions \cite{Ahookhosh2025}.

However, in many practical applications, the objective function is of a composite form; i.e., $\gf$ can be expressed as a composite function
$\gf(x) := f(x) + g(x)$, where $f:\R^n\to\R$ is a continuously differentiable (possibly nonconvex) function, and $g:\R^n\to\Rinf$ is a proper lower semicontinuous (possibly nonsmooth and nonconvex) function. Accordingly, the forward-backward envelope (FBE) was introduced in \cite{patrinos2013proximal} and further analyzed in \cite{Stella17,Themelis18, themelis2019acceleration} as an effective smoothing tool, given by
\begin{equation}\label{eq:fbsGeneral}
\fgam{\gf}{}{\gamma}(x):=\mathop{\bs{\inf}}\limits_{y\in \R^n} \left\{f(x)+\langle \nabla f(x), y-x\rangle+g(y)+\frac{1}{2\gamma}\Vert x- y\Vert^2\right\},
\end{equation}
where $f$ is $L$-smooth with $L>0$ and $\gamma\in(0,1/L]$. Many important properties of FBE, such as the preservation of the same infimum value as the original function $\gf$, depend on $L$-smoothness of $f$. In this case,  the classical \textit{descent lemma}~\cite[Proposition A.24]{bertsekas1999nonlinear}
applies to the smooth part $f$, namely
\begin{equation}\label{eq:lipdesc}
f(y)\leq \mathcal{M}(x,y):= f(x)+\langle \nabla f(x), y-x\rangle +\frac{L}{2}\Vert y - x\Vert^2,\qquad \forall x, y\in \R^n.
\end{equation}
Consequently, by adding $g(y)$ to both sides of \eqref{eq:lipdesc}, one obtains the majorization
$\gf(y)=f(y)+g(y)\le \mathcal{M}(x,y)+g(y)$,
so that the function
$\psi(x,y):=  \mathcal{M}(x,y)+g(y)$ acts as a majorant of the composite objective $\gf$.
This provides a suitable framework for developing splitting methods for composite optimization problems
\cite{Ahookhosh24,ahookhosh2021high,beck2017first,bonettini2025linesearch,combettes2005signal,lange2016mm,Stella17,sun2016majorization,Themelis18,Nesterov2023a}.

We note that $L$-smoothness of $f$ excludes many practical cases in which $f$ is only \textit{weakly smooth}, that is, when $\nabla f$ is H\"{o}lder continuous with exponent $\nu\in(0,1]$ (see Definition~\ref{def:holdgrad}); see, e.g., 
\cite{ahookhosh2019accelerated,berger2020quality,Nesterov15univ,yashtini2016global}, leading to
\begin{equation}\label{eq:holdmajf}
f(y)\leq  f(x) + \langle \nabla f(x), y - x \rangle+\frac{L_{\nu}}{1 + \nu}\Vert x - y \Vert ^{1+\nu},\qquad \forall x, y\in \R^n.
\end{equation}
This result indicates that $f$ can be majorized by a non-quadratic term of degree $p=1+\nu>1$, extending the classical descent lemma and motivating the use of \textit{high-order regularization} \cite{Bredies09, Nesterov15univ}.
Such generalized models have recently attracted attention in the context of weakly smooth optimization frameworks~\cite{Rodomanov2020,Cartis2017Worst,Nesterov15univ}.
In \cite{kabgani2025fb}, we studied the {\it high-order descent lemma}, stated as
\begin{equation}\label{eq:highOrdmajf}
f(y)\leq  f(x) + \langle \nabla f(x), y - x \rangle+\frac{L_p}{p}\Vert x - y \Vert ^{p},\qquad \forall x, y\in \R^n,
\end{equation}
for $p>1$, which was used to introduce a {\it high-order majorization-minimization framework}. 
Building upon a broader concept of the high-order majorant that generalizes beyond the H\"{o}lderian model in \eqref{eq:holdmajf}, we introduce the \textit{high-order forward–backward splitting} (HiFBS) mapping and the \textit{high-order forward–backward envelope} (HiFBE); see Definition~\ref{def:HFBE}.
These constructions enable the design of algorithms with regularization terms of arbitrary order~$p>1$, 
thereby covering composite models for which the standard quadratic forward-backward theory is not applicable.

Relaxing the $L$-smoothness assumption is also important from an application viewpoint. In several problems from signal processing and machine learning, the data-fidelity term is naturally weakly smooth, while its gradient is not globally Lipschitz continuous. A representative example is the regularized inverse problem
\[
\mathop{\bs{\min}}\limits_{x\in \R^n}\gf(x):=\frac{1}{q}\|Ax-b\|_q^q+g(x),
\]
with $q\in(1,2)$,
which arises in robust signal reconstruction and sparse recovery under impulsive or heavy-tailed noise. In this case, the smooth part $f(x):=\frac{1}{q}\|Ax-b\|_q^q$
has a $(q-1)$-H\"older continuous gradient, whereas $\nabla f$ is, in general, not Lipschitz continuous when $q<2$; see, e.g., \cite[Example~22]{kabgani2025fb}. Such subquadratic losses are also standard in robust regression and learning models, precisely because they are less sensitive to outliers than the quadratic loss. 
For a nonconvex example, consider a DC-type subquadratic model of the form
\[
f(x):=\frac{1}{q}\|Ax-b\|_q^q-\frac{\lambda}{q}\|Cx-d\|_q^q,
\]
where $q\in(1,2)$ and $\lambda>0$. This function remains continuously differentiable with a $(q-1)$-H\"older continuous gradient.
Therefore, the restriction to the classical quadratic majorization is too narrow for these applications, while the high-order model in \eqref{eq:holdmajf}, and more generally in \eqref{eq:highOrdmajf}, remains applicable with $p=1+\nu$. This provides a concrete motivation for studying HiFBS and HiFBE beyond the $L$-smooth setting.

While the algorithmic properties of HiFBS and HiFBE have been investigated in~\cite{kabgani2025fb}, several fundamental analytical aspects of HiFBS and HiFBE remain unexplored.  
In particular, the questions of  \textit{single-valuedness} and  \textit{continuity} of HiFBS, as well as the \textit{differentiability} of HiFBE, have not yet been characterized.  
These properties are central to understanding the behavior of high-order envelopes, linking them to high-order Moreau envelopes~\cite{Kabgani24itsopt,Kabganidiff} and to the convergence of high-order proximal methods.  
The results presented in this paper provide new insights into the analytical structure of the high-order forward-backward envelope, thus bridging variational analysis, smooth approximation, and high-order optimization theory. Consequently, this development facilitates the application of HiFBE in the design of gradient-based first-order methods.

\vspace{-4mm}
\subsection{{\bf Contribution}}\label{sec:contribution}

The contributions of this paper are summarized as follows:

\begin{description}[wide, labelwidth=!, labelindent=0pt]
 
     \item[{\bf(i)}] {\bf  Fundamental properties of HiFBS and HiFBE.} 
        We establish a rigorous framework for HiFBS and its associated envelope, HiFBE. In particular, we prove the boundedness and uniform boundedness of HiFBS on bounded sets and derive the H\"{o}lder and Lipschitz continuity properties of HiFBE. These results form the analytical foundation for our subsequent developments on differentiability and weak smoothness.
    \item[{\bf(ii)}] {\bf Differentiability and weak smoothness of HiFBE.}  
        We derive explicit expressions for the Fr\'{e}chet (regular) and Mordukhovich (limiting) subdifferentials of HiFBE and establish necessary and sufficient conditions for its differentiability. By leveraging the prox-regularity of $g$ and the concept of $p$-calmness, we demonstrate the local single-valuedness, continuity, and H\"{o}lder continuity of HiFBS. These properties, in turn, ensure the differentiability of HiFBE in a neighborhood of $p$-calm points. Furthermore, under mild assumptions regarding the smoothness of $f$ and the prox-boundedness of $g$, we prove the weak smoothness of HiFBE and explicitly quantify its H\"{o}lder continuity order.

\end{description}

\subsection{{\bf Organization}}\label{sec:organization}
The remainder of this paper is organized as follows. 
Section~\ref{sec:prelim} introduces the necessary preliminaries and notation. 
In Section~\ref{sec:hifbe}, after recalling the definitions of HiFBS and HiFBE, we study the boundedness and uniform boundedness of HiFBS, as well as the H\"{o}lder and Lipschitz continuity properties of HiFBE. 
Section~\ref{sec:diff} studies the subdifferential structure, differentiability, and weak smoothness of HiFBE, together with the local single-valuedness and continuity of HiFBS. 
Finally, Section~\ref{sec:disc} concludes the paper and discusses possible directions for future research.



\section{Preliminaries and notation}
\label{sec:prelim}

Throughout this paper, $\R^n$ denotes the $n$-dimensional \textit{Euclidean space}, 
while $\Vert \cdot \Vert$ and $\langle \cdot, \cdot \rangle$ represent the \textit{Euclidean norm}
 and \textit{inner product}, respectively. 
We denote the set of \textit{natural numbers} by $\mathbb{N}$ and let $\Nz := \mathbb{N} \cup \{0\}$. 
The set $\mb(\ov{x}; r)$ is the \textit{open ball} centered at $\ov{x} \in \R^n$ with radius $r > 0$. 
The \textit{interior} and \textit{closure} of a set $C \subseteq \R^n$ are denoted by $\inter{C}$ and 
$\closure{C}$, respectively. 
We adopt the convention $\infty - \infty = \infty$.
The \textit{effective domain} of $\gh: \R^n \to \Rinf := \R \cup \{+\infty\}$ is $\dom{\gh} := \{x \in \R^n \mid \gh(x) < +\infty\}$,
and $\gh$ is called \textit{proper} if $\dom{\gh} \neq \emptyset$. 
The set of \textit{minimizers} of $\gh$ over $C \subseteq \R^n$ is denoted by $\argmin{x \in C} \gh(x)$. 
The function $\gh$ is \textit{lower semicontinuous} (lsc) at $\ov{x} \in \R^n$ if $\bs\liminf_{k \to \infty} \gh(x^k) \geq \gh(\ov{x})$ for every sequence $x^k \to \ov{x}$. 
For a set-valued map $\Psi: \R^n \rightrightarrows \R^n$, the \textit{domain} is defined as $\Dom{\Psi} := \{x \in \R^n \mid \Psi(x) \neq \emptyset\}$. 

If $p > 1$, the gradient of $\frac{1}{p} \Vert x \Vert^p$ is $\nabla \left(\frac{1}{p} \Vert x \Vert^p \right) = \Vert x \Vert^{p-2} x$, 
where the convention $\frac{0}{0} = 0$ is adopted for $x = 0$. For a sequence $\{x^k\}_{k\in \Nz}$, 
 $\ov{x}$ is the \textit{limit} of the sequence if
$x^k\to \ov{x}$, and $\widehat{x}$ is a \textit{cluster point} of this sequence if  there exist an infinite subset $J\subseteq \Nz$ and a subsequence $\{x^{j}\}_{j\in J}$ such that $x^j\to \widehat{x}$.

The following elementary relations will be used repeatedly in the sequel to bound higher-order terms and to establish the monotonicity-type inequalities.
\begin{lemma}[Basic inequalities]\label{lem:findlowbounknu:lemma}
Let $a, b\in \R^n$.
\begin{enumerate}[label=(\textbf{\alph*}), font=\normalfont\bfseries, leftmargin=0.7cm]
\item\label{lem:basicineq:a} For each $p\in (1,2]$, there exists some $\kappa_p>0$ such that for any  $r>0$ and $a, b\in \mb(0; r)$,
\[
\langle \Vert a\Vert^{p-2}a - \Vert b\Vert^{p-2}b, a-b\rangle\geq \kappa_pr^{p-2}\Vert a - b\Vert^2.
\]
\item\label{lem:basicineq:b} For each $p\geq 1$ and $a, b\in \R^n$, 
\begin{equation}\label{eq:baseq:pgen}
\Vert a-b\Vert^p\leq 2^{p-1}\left(\Vert a\Vert^p+\Vert b\Vert^p\right).
\end{equation}
\end{enumerate}
\end{lemma}
\begin{proof}
    For Assertion~\ref{lem:basicineq:a}, see \cite[Lemma 2]{Kabganidiff}, and for  Assertion~\ref{lem:basicineq:b}, see \cite[Lemma 2.1]{Kabgani24itsopt}.
\end{proof}

Since the analysis of HiFBE involves nonsmooth and possibly nonconvex components, we recall the two main notions of generalized gradients used throughout the paper.

\begin{definition}[Fr\'{e}chet and Mordukhovich subdifferentials]\label{def:subdiff}
The \textit{Fr\'{e}chet/regular} and \textit{Mordukhovich/limiting subdifferentials} of a function $\gh: \R^n\to\Rinf$ at $\ov{x}\in \dom{\gh}$ are defined, respectively, as~\cite{Mordukhovich2018,Rockafellar09}
\[
\widehat{\partial}\gh(\ov{x}):=\left\{\zeta\in \R^n\mid~\mathop{\bs\liminf}\limits_{x\to \ov{x}}\frac{\gh(x)- \gh(\ov{x}) - \langle \zeta, x - \ov{x}\rangle}{\Vert x - \ov{x}\Vert}\geq 0\right\},
\]
and
\[
\partial \gh(\ov{x}):=\left\{\zeta\in \R^n\mid~\exists x^k\to \ov{x}, \zeta^k\in \widehat{\partial}\gh(x^k),~\text{with}~\gh(x^k)\to \gh(\ov{x})~\text{and}~ \zeta^k\to \zeta\right\}.
\]
\end{definition}
We have $\widehat{\partial}\gh(\ov{x})\subseteq\partial \gh(\ov{x})$.
These subdifferentials will be used 
to characterize the differentiability of HiFBE.
Next, we recall smoothness notions.
A function $\gh: \R^n \to \R$  is  \textit{Fr\'{e}chet differentiable} at $\ov{x}\in \inter{\dom{\gh}}$ 
with \textit{Fr\'{e}chet derivative}  
$\nabla \gh(\ov{x})$
 if 
\[
\mathop{\bs\lim}\limits_{x\to \ov{x}}\frac{\gh(x) -\gh(\ov{x}) - \langle \nabla \gh(\ov{x}) , x - \ov{x}\rangle}{\Vert x - \ov{x}\Vert}=0.
\]
For a set $C\subseteq\R^n$, the notation $\gh\in \mathcal{C}^{k}(C)$ indicates that $\gh$ is $k$-times continuously differentiable on $C$, where $k\in \mathbb{N}$. 

\begin{definition}[H\"{o}lder property]\label{def:hold}
A proper function $\gh: \R^n \to \Rinf$ has a
\textit{$\nu$-H\"{o}lder property} with $\nu \in (0,1]$ on a set $C\subseteq \dom{\gh}$ if there exists some $L_\nu > 0$ such that 
\[
\vert\gh(x) - \gh(y)\vert \leq L_\nu \| x - y \|^\nu, \qquad \forall x, y \in C.
\]
\end{definition}
The H\"{o}lder property serves as a basic building block for defining weak smoothness and generalizing the Lipschitz continuity of the gradient.
The class of functions with H\"{o}lder continuous gradients has recently garnered increased attention for its applications in optimization \cite{berger2020quality,bolte2023backtrack,Cartis2017Worst,Nesterov15univ,yashtini2016global}.

\begin{definition}[H\"{o}lder continuous gradient]\label{def:holdgrad}
A function $\gh: \R^n \to \R$ has a \textit{$\nu$-H\"{o}lder continuous gradient} on $\R^n$ with $\nu \in (0, 1]$ if it is Fr\'{e}chet differentiable and 
there exists a constant $L_\nu\geq 0$ such that
\begin{equation}\label{eq:nu-Holder continuous gradient}
\Vert \nabla \gh(y)- \nabla \gh(x)\Vert \leq L_\nu \Vert y-x\Vert^\nu, \qquad \forall x, y\in \R^n.
\end{equation}
\end{definition}
The class of functions with $\nu$-H\"{o}lder continuous gradients is denoted by $\mathcal{C}^{1, \nu}_{L_\nu}(\R^n)$ and referred to as \textit{weakly smooth}, and if $\nu=1$, $\gh$ is called \textit{$L$-smooth}. 

The following H\"olderian descent lemma is stated in \cite[Lemma~1]{yashtini2016global} for a function $\gh\in \mathcal{C}^{1, \nu}_{L_{\nu}}(\R^n)$. 
\begin{fact}[H\"olderian descent lemma]\label{fact:holder:declem}\cite{yashtini2016global}
Let $\gh\in \mathcal{C}^{1, \nu}_{L_{\nu}}(\R^n)$ with $\nu\in (0, 1]$. Then, for all $x, y\in \R^n$, 
\begin{equation}\label{eq:upperbounf for c1,alp}
\left\vert\gh(y)- \gh(x) - \langle \nabla \gh(x), y - x \rangle\right\vert\leq\frac{L_{\nu}}{1 + \nu}\Vert y-x \Vert ^{1+\nu}.
\end{equation}
\end{fact}

We continue with the following useful lemma, which follows directly from Fact~\ref{fact:holder:declem} by applying it to the pairs $(x,y)$ and $(y,x)$ and then using the triangle inequality.

\begin{lemma}[H\"older estimation]\label{lem:hol:oneside} 
Let $\gh\in \mathcal{C}^{1, \nu}_{L_{\nu}}(\R^n)$ with $\nu\in (0, 1]$. Then, for each $x,y\in \R^n$,
\[
\left\vert\langle \nabla \gh(y)-\nabla \gh(x), y-x \rangle\right\vert\leq \frac{2L_\nu}{1+\nu} \Vert x - y \Vert ^{1+\nu}.
\]
\end{lemma}

The next concept captures a broad class of nonsmooth functions that admit quadratic local approximations
and will play an important role in the later analysis of HiFBS.

 \begin{definition}[Prox-regularity]\label{def:prox-regular}\cite{Poliquin96}
Let $\gh: \R^n\to \Rinf$ be a proper lsc function and let $\ov{x}\in\dom{\gh}$. Then, $\gh$ is said to be \textit{prox-regular} at $\ov{x}$ for $\ov{\zeta}\in\partial \gh(\ov{x})$ if there exist $\varepsilon>0$ and $\rho>0$ such that
\[
\gh(x')\geq \gh(x)+\langle \zeta, x'-x\rangle-\frac{\rho}{2}\Vert x'-x\Vert^2,~\qquad \forall x'\in \mb(\ov{x}; \varepsilon),
\]
whenever $x\in \mb(\ov{x}; \varepsilon)$, $\zeta\in \partial \gh(x)\cap \mb(\ov{\zeta}; \varepsilon)$, and 
$\vert \gh(x)-\gh(\ov{x})\vert<\varepsilon$.
\end{definition}

We now recall the notions of high-order proximal operators and their envelopes. 
These constructions are the analytical backbone of HiFBS and HiFBE.
We briefly review these concepts and related properties.

\begin{definition}[High-order proximal operator and Moreau envelope]\label{def:home}
Let $p>1$, $\gamma>0$, and $\gh: \R^n \to \Rinf$ be a proper function. 
The \textit{high-order proximal operator} (\textit{HOPE}) of $\gh$ with parameter $\gamma$,
$\prox{\gh}{\gamma}{p}: \R^n \rightrightarrows \R^n$, is 
    \begin{equation}\label{eq:Hiorder-Moreau prox}
       \prox{\gh}{\gamma}{p} (x):=\argmint{y\in \R^n} \left\{\gh(y)+\frac{1}{p\gamma}\Vert x- y\Vert^p\right\},
    \end{equation}     
and the \textit{high-order Moreau envelope} (\textit{HOME}) of $\gh$ with parameter $\gamma$, 
$\fgam{\gh}{\gamma,p}{}: \R^n\to \R\cup\{\pm \infty\}$, 
is 
    \begin{equation}\label{eq:Hiorder-Moreau env}
    \fgam{\gh}{\gamma,p}{}(x):=\mathop{\bs{\inf}}\limits_{y\in \R^n} \left\{\gh(y)+\frac{1}{p\gamma}\Vert x- y\Vert^p\right\}.
    \end{equation}
    \end{definition}

We recall the notion of high-order prox-boundedness from \cite{Kabgani24itsopt}.
\begin{definition}[High-order prox-boundedness]\label{def:s-prox-bounded}\cite[Definition~3.3]{Kabgani24itsopt}
A proper function $\gh:\R^n\to \Rinf$ is \textit{high-order prox-bounded} with order $p>1$ if there exist
$\gamma>0$ and $x\in \R^n$ such that
$\fgam{\gh}{\gamma,p}{}(x)>-\infty$. 
The supremum of all such $\gamma$ is denoted by $\gamma^{\gh, p}$ and is referred to as the threshold of high-order prox-boundedness for $\gh$.
\end{definition}

This property guarantees that the high-order proximal operator and envelope are well defined.
If $\gh:\R^n\to \Rinf$ is convex or bounded below, it satisfies Definition~\ref{def:s-prox-bounded} with $\gamma^{\gh, p}=+\infty$ \cite{Kabgani24itsopt}.

The next fact, given in \cite[Theorem~3.4]{Kabgani24itsopt}, ensures that HOPE and HOME are well defined under the prox-boundedness condition and will be used later to establish the existence and compactness of HiFBS solutions.
Although \cite[Theorem~3.4]{Kabgani24itsopt} is stated under a broader assumption, only the high-order prox-boundedness component of that theorem is needed here.
 \begin{fact}[Well-definedness of HOME and HOPE]\label{fact:level-bound+locally uniform}\cite[Theorem~3.4]{Kabgani24itsopt}
Let $p>1$ and $\gh: \R^n\to \Rinf$ be a proper lsc function that is high-order prox-bounded with threshold $\gamma^{\gh, p}>0$. For each $\gamma\in (0, \gamma^{\gh, p})$ and $x\in \R^n$,  the set $\prox{\gh}{\gamma}{p}(x)$ is nonempty and compact, and $\hfgam{\gh}{p}{\gamma}(x)$ is finite.
 \end{fact}



\section{High-order forward-backward splitting mapping and envelope}\label{sec:hifbe}

In this section, we recall the \textit{high-order forward–backward splitting mapping} (HiFBS) and the associated \textit{high-order forward–backward envelope} (HiFBE), which are central to our analysis. After recalling essential properties and proving boundedness of HiFBS, we study the H\"{o}lder and Lipschitz properties of HiFBE in Subsection~\ref{sub:Holpro}. Moreover, by recalling the notion of \textit{calm points} and deriving their consequences, we obtain a \textit{uniform boundedness} property for HiFBS in Subsection~\ref{subsec:calm}. These results and notions are crucial for proving the differentiability and weak smoothness of HiFBE in the subsequent section.

The following constructions extend the classical FBE/FBS ($p=2$) to arbitrary orders $p>1$ and, in particular, cover weakly smooth models when $p=1+\nu$ with $\nu\in(0,1]$.
\begin{definition}[High-order forward-backward splitting mapping and envelope]\label{def:HFBE}\cite{kabgani2025fb}
Let $p> 1$, $\gamma>0$, $f:\R^n\to \R$ be Fr\'{e}chet differentiable, and let $g:\R^n\to \Rinf$ be a proper lsc function.
The \textit{high-order forward-backward splitting mapping} (HiFBS) of $\gf(x):=f(x)+g(x)$ with parameter $\gamma$, $ \Tprox{\gf}{\gamma}{p}: \R^n \rightrightarrows \R^n$, is defined as
\begin{equation}\label{HFBM}
\Tprox{\gf}{\gamma}{p} (x):=\argmint{y\in \R^n}\left\{f(x)+\langle \nabla f(x) , y - x\rangle +g(y)+\frac{1}{p\gamma}\Vert x-y\Vert^p\right\}.
\end{equation}
 The \textit{high-order forward-backward envelope} (HiFBE) of $\gf$ with parameter $\gamma$, $\fgam{\gf}{p}{\gamma}: \R^n\to \R\cup\{\pm \infty\}$, is defined as
\begin{equation}\label{HFBE}
 \fgam{\gf}{p}{\gamma}(x):=\mathop{\bs{\inf}}\limits_{y\in \R^n}\left\{f(x)+\langle \nabla f(x) , y - x\rangle +g(y)+\frac{1}{p\gamma}\Vert x-y\Vert^p\right\}.
\end{equation}
\end{definition}

Setting $p=2$ in the definition of HiFBE, one \textit{recovers} the classical forward–backward envelope; see \cite[Definition~2.1]{Stella17}. 
The set-valued residual mapping $R_{\gamma,p}:\R^n\rightrightarrows \R^n$ is defined by
$\bs R_{\gamma , p}(x):=x - \Tprox{\gf}{\gamma}{p} (x)$ and we set
\begin{equation}\label{eq:ellfun} 
  \ell(x,y):=f(x)+\langle \nabla f(x) , y - x\rangle +g(y).
\end{equation}
For $p=2$, we have the identity $\Tprox{\gf}{\gamma}{2} (x)= \prox{g}{\gamma}{2} (x-\gamma\nabla f(x))$ \cite[Eq.~(1.7)]{Stella17}, which aligns the forward–backward splitting mapping with the standard proximal operator. For $p\neq 2$, however, this relationship between HiFBS at $x$ for $\gf$ and HOPE at $x-\gamma\nabla f(x)$ for $g$ \textit{fails}, even under convexity; see \cite[Remark~24]{kabgani2025fb}. Consequently, results developed for HOME and HOPE in \cite{Kabgani24itsopt} cannot be directly applied to HiFBE and HiFBS and require a dedicated analysis.

In some of our results, we consider the following assumptions. 
\begin{assumption}[Assumptions on $f$]\label{assum:mainassum:f} 
The function
$f:\R^n\to \R$ is Fr\'{e}chet differentiable and satisfies
$f\in \mathcal{C}^{1, \nu}_{L_\nu}(\R^n)$ with $\nu\in (0,1]$. 
\end{assumption}

\begin{assumption}[Assumptions on $g$]\label{assum:mainassum:g} 
The function $g:\R^n\to \Rinf$ is proper, lsc, and
   high-order prox-bounded with the threshold $\gamma^{g, p}>0$.
\end{assumption}

The following Fact summarizes several basic properties of HiFBE that will be used throughout this section.
\begin{fact}[Fundamental properties of HiFBE]\label{th:basichifbe}\cite[Theorem~25]{kabgani2025fb}
Let $p>1$, $f:\R^n\to \R$ be Fr\'{e}chet differentiable, and $g:\R^n\to \Rinf$ be a proper lsc function. For $\gf(x):=f(x)+g(x)$, the following statements hold:
\begin{enumerate}[label=(\textbf{\alph*}), font=\normalfont\bfseries, leftmargin=0.7cm]
\item \label{th:basichifbe:dom} $\dom{\fgam{\gf}{p}{\gamma}}=\R^n$ 
and $\fgam{\gf}{p}{\gamma}(x)\leq  \gf(x)$  for each $\gamma>0$ and $x\in \R^n$.
\end{enumerate}
If Assumption~\ref{assum:mainassum:f} holds and $p=1+\nu$, then for each $\gamma\in (0, L_\nu^{-1})$, 
 \begin{enumerate}[label=(\textbf{\alph*}), font=\normalfont\bfseries, leftmargin=0.7cm, start=2]
\item \label{th:basichifbe:ineqforp} $\gf(y)\leq \ell(x,y)+\frac{1}{p\gamma}\Vert x - y\Vert^p$ for all $x,y\in \R^n$. Moreover, 
if $y\in \Tprox{\gf}{\gamma}{p} (x)$, then $\gf(y)\leq \fgam{\gf}{p}{\gamma}(x)$;
\item \label{th:basichifbe:infimforp} for each $\mu\in (\gamma, L_\nu^{-1})$, $\bs\inf_{z\in\R^n}\gf(z)\leq \fgam{\gf}{p}{\mu}(x)\leq  \fgam{\gf}{p}{\gamma}(x)\leq  \gf(x)$ for each $x\in \R^n$;
\item \label{th:basichifbe:infim2forp}$\mathop{\bs{\inf}}\limits_{z\in\R^n}\gf(z)=\mathop{\bs{\inf}}\limits_{z\in\R^n}\fgam{\gf}{p}{\gamma}(z)$.
\end{enumerate}
If Assumption~\ref{assum:mainassum:g} holds,  then
 \begin{enumerate}[label=(\textbf{\alph*}), font=\normalfont\bfseries, leftmargin=0.7cm, start=5]
\item \label{th:basichifbe:finite} 
for each $x\in \R^n$ and $\gamma\in (0, \gamma^{g, p})$, $\fgam{\gf}{p}{\gamma}(x)$ is finite;

 \item \label{th:basichifbe:levelunif} the function $\Psi(y, x, \gamma):=\ell(x,y)+\frac{1}{p\gamma}\Vert x- y\Vert^p$  is
 level-bounded in $y\in\R^n$ locally uniformly in $(x, \gamma)\in\R^n\times (0, \gamma^{g, p})$. 
Additionally, this function is lsc;

\item \label{th:basichifbe:con} $\fgam{\gf}{p}{\gamma}$ depends continuously on $(x,\gamma)$ in $\R^n\times (0, \gamma^{g, p})$.
 \end{enumerate}
 If there exist $\gamma>0$ and $\ov{x}\in \R^n$ such that $\fgam{\gf}{p}{\gamma}(\ov{x})>-\infty$, then
  \begin{enumerate}[label=(\textbf{\alph*}), font=\normalfont\bfseries, leftmargin=0.7cm, start=8]
 \item \label{th:basichifbe:finproxb} $g$ is high-order prox-bounded.
  \end{enumerate}
\end{fact}

The majorant given in Fact~\ref{th:basichifbe}~\ref{th:basichifbe:ineqforp} is based on the majorant \eqref{eq:holdmajf} for the function $f$.
Note that this type of majorant is generally not encompassed by the standard descent lemma~\eqref{eq:lipdesc}. 
For instance, consider the function $f: \R \to \R$ defined by $f(x) = \frac{1}{q} |x|^q$ with $q = \frac{3}{2}$, which exhibits a $\frac{1}{2}$-H\"{o}lder continuous gradient. 
By setting $x = 0$, \eqref{eq:lipdesc} simplifies to $|y|^q \leq \frac{q L}{2} |y|^2$, an inequality that does not hold for all $y \in \R$. 
Moreover, even if \eqref{eq:holdmajf} holds for some $\nu \in (0, 1]$, it may fail for other exponents, i.e., one may not have $f(y)\leq  \mathcal{M}_\mu(x,y)$ for all $y\in \R^n$ whenever $\mu\in(0, 1]$ and $\mu\neq\nu$,
where
$\mathcal{M}_\nu(x,y):= f(x) + \langle \nabla f(x), y - x \rangle+\frac{L_{\nu}}{1 + \nu}\Vert x - y \Vert ^{1+\nu}$.
Figures~\ref{fig:maj521}~(a)~and~(b) illustrate this behavior for $\gf(x)=\frac{2}{3}|x|^{\frac{3}{2}} - 0.5\cos(3x) + 0.2|x|$ evaluated at $\ov{x}=0.5$. 
Figure~\ref{fig:maj521}~(a) shows that $\mathcal{M}_{0.5}(x,y)$ successfully majorizes $\gf(x)$, whereas $\mathcal{M}_{1}(x,y)$ fails, dipping below $\gf(x)$. 
Figure~\ref{fig:maj521}~(b) confirms that $\mathcal{M}_{0.2}(x,y)$ also ceases to be a valid majorant.

\begin{figure}
        \subfloat[$\mathcal{M}_{0.5}(x,y)$ and $\mathcal{M}_{1}(x,y)$]
        {\includegraphics[width=0.4\textwidth]{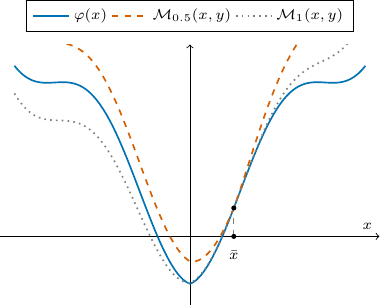}}\qquad\qquad\qquad\qquad
        \subfloat[$\mathcal{M}_{0.5}(x,y)$ and $\mathcal{M}_{0.2}(x,y)$]
        {\includegraphics[width=0.4\textwidth]{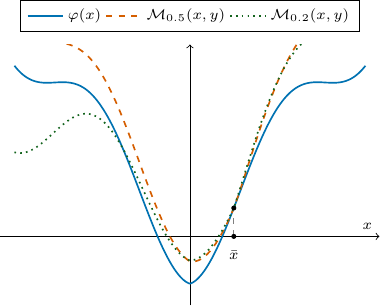}}
        \vspace{2mm}
        \caption{
        Graphs of $\gf$, $\mathcal{M}_{0.2}(x,y)$,  $\mathcal{M}_{0.5}(x,y)$, and $\mathcal{M}_{1}(x,y)$. \label{fig:maj521}}
    \end{figure}

The next example illustrates how the choice of $\gamma$ affects the shape of the HiFBE and validates Fact~\ref{th:basichifbe}~\ref{th:basichifbe:infim2forp}.
\begin{example}\label{ex:diffgam}
Consider $f,g: \R\to\R$ defined by $f(x) = 0.5 \vert x\vert^{\frac{3}{2}}$ and $g(x) = \vert 0.3 \sin(5x)\vert + 0.2 x^2 e^{-x^2}$
and let $\gf(x) = f(x) + g(x)$. Here, $f$ exhibits a  $\frac{1}{2}$-H\"{o}lder continuous gradient, while $g$ is nonconvex and oscillatory, introducing multiple local minima into $\gf$. We set $p =1 + \nu= 1.5$, consistent with the H\"{o}lder exponent $\nu = \frac{1}{2}$.
HiFBE, $\fgam{\gf}{p}{\gamma}$,  is computed numerically on a grid over $x \in [-2.5, 2.5]$. 
Figure~\ref{fig:ex:diffgam} displays $\gf$ and $\fgam{\gf}{p}{\gamma}$ for three values of $\gamma$.

Choosing a large $\gamma$, here $\gamma = 2$, may violate the condition $\gamma\in (0, L_\nu^{-1})$, breaking the equivalence in Fact~\ref{th:basichifbe}~\ref{th:basichifbe:infim2forp}. 
As shown in Figure~\ref{fig:ex:diffgam}~(a), a large $\gamma$ weakens the penalty term $\frac{1}{p\gamma}\Vert x - y\Vert^p$, allowing the minimizer $y\in \Tprox{\gf}{\gamma}{p}(x)$ to deviate substantially from $x$, and thus $\fgam{\gf}{p}{\gamma}$ diverges from $\gf$. 
Conversely, Figure~\ref{fig:ex:diffgam}~(b) illustrates that a small $\gamma = 0.2$ strengthens the penalty excessively, making $\fgam{\gf}{p}{\gamma}\approx \gf$. 
While this preserves fidelity to $\gf$, it also inherits its nonsmooth and multimodal behavior, which may hinder optimization. 
A moderate value, such as $\gamma = 1$ (Figure~\ref{fig:ex:diffgam}~(c)), provides a balance: it preserves the equivalence in Fact~\ref{th:basichifbe}~\ref{th:basichifbe:infim2forp}, produces a smoother envelope, and mitigates spurious local minima. 
This example highlights the critical role of $\gamma$ in the shaping of HiFBE and its influence on the balance between fidelity and smoothness in nonconvex optimization.
\end{example}

\begin{figure}
        \subfloat[ Graphs of $\gf$ and $\fgam{\gf}{p}{\gamma}$: $\gamma = 2$ ]
        {\includegraphics[width=5.5cm]{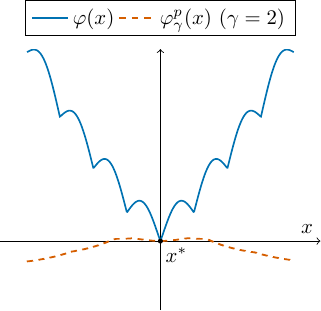}}\quad
           \subfloat[ Graphs of $\gf$ and $\fgam{\gf}{p}{\gamma}$: $\gamma = 0.2$ ]
        {\includegraphics[width=5.5cm]{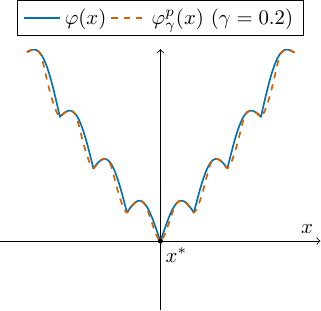}}\quad
       \subfloat[ Graphs of $\gf$ and $\fgam{\gf}{p}{\gamma}$: $\gamma = 1$ ]
       {\includegraphics[width=5.5cm]{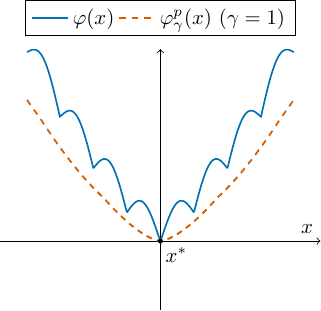}}
        \caption{
        Graphs of $\gf$ and $\fgam{\gf}{p}{\gamma}$ for different values of $\gamma$ in Example~\ref{ex:diffgam}.
         \label{fig:ex:diffgam}}
    \end{figure}

The following fact summarizes key existence and stability properties of the HiFBS.
\begin{fact}[Fundamental properties of HiFBS]\label{th:hopbfb}\cite[Theorem~27]{kabgani2025fb}
Let $p>1$, Assumption~\ref{assum:mainassum:g} hold, and let 
$f:\R^n\to \R$ be Fr\'{e}chet differentiable. For the composite function $\gf(x):=f(x)+g(x)$, the following statements hold:
 \begin{enumerate}[label=(\textbf{\alph*}), font=\normalfont\bfseries, leftmargin=0.7cm]
\item \label{th:hopbfb:proxb:proxnonemp} For each $x\in \R^n$ and $\gamma\in (0, \gamma^{g, p})$, the set
 $\Tprox{\gf}{\gamma}{p}(x)$ is nonempty and compact;

\item \label{th:hopbfb:proxb:conv} If $y^k\in \Tprox{\gf}{\gamma^k}{p}(x^k)$, $x^k\to \ov{x}$, and $\gamma^k\to \gamma\in (0, \gamma^{g, p})$, then the sequence $\{y^k\}_{k\in \mathbb{N}}$ is bounded and all of its cluster points lie in $\Tprox{\gf}{\gamma}{p}(\ov{x})$.
    \end{enumerate} 
\end{fact}

The next lemma provides a lower-boundedness property that will be instrumental in proving the boundedness of HiFBS and the H\"{o}lder continuity of HiFBE in the following subsections.
\begin{lemma}[Lower-boundedness]\label{lem:lower}
Let $p>1$, Assumption~\ref{assum:mainassum:g} hold and let 
$f:\R^n\to \R$ be Fr\'{e}chet differentiable.
Then, for each $r>0$ and $\gamma\in (0, \gamma^{g, p})$, the function $(x,y)\mapsto \ell(x,y)+\frac{2^{p-1}}{p\gamma}\Vert y\Vert^p$
is bounded from below on $\mb(0; r)\times\R^n$. 
\end{lemma}
\begin{proof}
Fix $r>0$ and $\gamma\in (0, \gamma^{g, p})$. By Fact~\ref{th:basichifbe}~\ref{th:basichifbe:finite}~and~\ref{th:basichifbe:con}, for $\gf(x):=f(x)+g(x)$,  the mapping
$x\mapsto \fgam{\gf}{p}{\gamma}(x)$ is finite and continuous on $\R^n$.   Hence, on the compact set $\closure{\mb(0;r)}$, it is bounded from below. Let
\[
\beta_r:=\mathop{\bs{\inf}}\limits_{u\in\closure{\mb(0;r)}} \fgam{\gf}{p}{\gamma}(u)>-\infty.
\]
For any $x\in \mb(0; r)$ and $y\in\R^n$, we obtain
\begin{align*}
\beta_r\leq \fgam{\gf}{p}{\gamma}(x)&\overset{(i)}{\leq} f(x)+\langle \nabla f(x) , y - x\rangle+g(y)+\frac{1}{p\gamma}\Vert x- y\Vert^p
\\&\overset{(ii)}{\leq} f(x)+\langle \nabla f(x) , y - x\rangle+g(y)+\frac{2^{p-1}}{p\gamma}\left(\Vert x\Vert^p+\Vert y\Vert^p\right)
\\&\leq f(x)+\langle \nabla f(x) , y - x\rangle+g(y)+\frac{2^{p-1}}{p\gamma}\left(r^p+\Vert y\Vert^p\right),
\end{align*}
where $(i)$ follows from the definition of $\fgam{\gf}{p}{\gamma}(x)$ and $(ii)$ from the inequality \eqref{eq:baseq:pgen}, i.e., 
\[
\beta_r - \frac{2^{p-1}}{p\gamma}r^p\leq 
f(x)+\langle \nabla f(x) , y - x\rangle+g(y)+\frac{2^{p-1}}{p\gamma}\Vert y\Vert^p, \qquad \forall (x,y)\in \mb(0;r)\times \R^n.
\]
This proves the claim.
\end{proof}

\subsection{{\bf H\"{o}lder property of HiFBE}}
\label{sub:Holpro}
Before establishing the H\"{o}lder continuity of HiFBE, we first show that the HiFBS mapping is bounded on bounded sets. This ensures that minimizers of the forward–backward subproblem remain confined to a compact region, which is essential for subsequent regularity analysis.
We show that for each $r>0$ and an appropriately chosen $\gamma$, 
there exists $\tau>0$ such that $\Tprox{\gf}{\gamma}{p}\left(\mb(0; r)\right)\subseteq\mb(0; \tau)$.
\begin{proposition}[Boundedness of HiFBS]\label{prop:findtau:lip} 
Let Assumptions~\ref{assum:mainassum:f}~and~\ref{assum:mainassum:g} hold, let $p=1+\nu$, and define $\gf:\R^n\to\Rinf$ by $\gf(x):=f(x)+g(x)$. 
Then there exists $\ov{\gamma}\in (0, \gamma^{g,p})$ such that for each $\gamma\in (0,\ov{\gamma})$ and any $r>0$, there exists a finite $\tau>0$ satisfying
\begin{equation*}\label{findtau:main}
\Tprox{\gf}{\gamma}{p}(x)=\argmint{\{y\in \R^n \mid \Vert y\Vert \leq\tau\}}\left\{\ell(x,y)+\frac{1}{p\gamma}\Vert x- y\Vert^p\right\},\qquad \forall\,x\in \mb(0;r).
\end{equation*}
\end{proposition}
\begin{proof}
Let $\widehat{\gamma}\in(0,\gamma^{g,p})$ be arbitrary and choose $\gamma\in(0,4^{1-p}\widehat{\gamma})$.  
From Lemma~\ref{lem:lower}, together with $c_0:=\tfrac{2^{p-1}}{p\widehat{\gamma}}$, there exists $c_1\in\R$ such that the mapping $(x,y)\mapsto\ell(x,y)+c_0\Vert y\Vert^p$ is bounded from below on $\mb(0;r)\times\R^n$ by $c_1$.  
For any $x\in \mb(0;r)$ and any $y\in\Tprox{\gf}{\gamma}{p}(x)$, setting $\ov{x}=0$, we have
 \begin{align*}\label{findtau:eq2}
-c_0\Vert y\Vert^p +c_1+\frac{1}{p\gamma}\Vert x- y\Vert^p&\leq  \ell(x,y)+\frac{1}{p\gamma}\Vert x- y\Vert^p=\fgam{\gf}{p}{\gamma}(x)
\\&\leq f(x)+\langle \nabla f(x), \ov{x}- x\rangle+g(\ov{x})+\frac{1}{p\gamma}\Vert x- \ov{x}\Vert^p
\\&\overset{(i)}{\leq} f(\ov{x})+\langle \nabla f(\ov{x}), x- \ov{x}\rangle+\frac{L_\nu}{p}\Vert x- \ov{x}\Vert^{p}+\langle \nabla f(x), \ov{x}- x\rangle+g(\ov{x})+\frac{1}{p\gamma}\Vert x- \ov{x}\Vert^p
\\
&= f(\ov{x})+\langle \nabla f(\ov{x})-\nabla f(x), x- \ov{x}\rangle+g(\ov{x})+\frac{L_\nu}{p}\Vert x- \ov{x}\Vert^{p}+\frac{1}{p\gamma}\Vert x- \ov{x}\Vert^p,
 \end{align*}
 where $(i)$ follows from the H\"{o}lderian descent lemma (Fact~\ref{fact:holder:declem}).  
By Lemma~\ref{lem:hol:oneside},
 \[
 -c_0\Vert y\Vert^p +c_1+\frac{1}{p\gamma}\Vert x- y\Vert^p\leq \gf(\ov{x})+  \frac{2L_\nu}{p}\Vert  x- \ov{x}\Vert^p+\frac{L_\nu}{p}\Vert x- \ov{x}\Vert^{p}+\frac{1}{p\gamma}\Vert x- \ov{x}\Vert^p.
 \]
Using \eqref{eq:baseq:pgen},
we obtain
\[
 -c_0\Vert y\Vert^p+\frac{1}{p\gamma}\left(2^{1-p}\Vert y\Vert^p-\Vert x\Vert^p\right)\leq\frac{3\gamma  L_\nu+1}{p\gamma}\Vert x\Vert^p+(\gf(\ov{x})-c_1).
\]
Hence, 
\[
(2^{1-p}-c_0 p\gamma)\Vert y\Vert^p\leq (3\gamma  L_\nu+2)\Vert x\Vert^p+p\gamma(\gf(\ov{x})-c_1).
\]
Since $\gamma<4^{1-p}\widehat{\gamma}=\tfrac{2^{1-p}}{c_0p}$, we conclude that
\begin{equation}\label{eq:prop:findtau:lip}
\Vert y\Vert\leq\left(\frac{(3\gamma  L_\nu+2)r^p+p\gamma(\gf(\ov{x})-c_1)}{2^{1-p}-c_0 p\gamma}\right)^{\frac{1}{p}}=:\tau.
\end{equation}
Setting $\ov{\gamma}=4^{1-p}\widehat{\gamma}$, our desired result holds.
 \end{proof}
The following remark regarding Proposition~\ref{prop:findtau:lip} clarifies the dependence of the value of $\tau$ on $\gamma$ and $r$, and explains how it can be made independent of $\gamma$.
\begin{remark}\label{rem:ontau}
Proposition~\ref{prop:findtau:lip} shows that for any fixed $\widehat{\gamma}\in(0,\gamma^{g,p})$ and $\gamma\in(0,4^{1-p}\widehat{\gamma})$, there exists $\tau=\tau(\gamma,r)$ such that $\Tprox{\gf}{\gamma}{p}(x)\subseteq\mb(0;\tau)$ for all $x\in\mb(0;r)$, where $\tau$ is given by~\eqref{eq:prop:findtau:lip}.  
In this formula, $\tau$ depends on both $\gamma$ and $r$.  
Since 
\[
\Gamma(\gamma):=2^{1-p}-c_0p\gamma
=2^{1-p}-\frac{2^{p-1}\gamma}{\widehat{\gamma}},
\]
we have $\Gamma(\gamma)\to 0$ as $\gamma\uparrow4^{1-p}\widehat{\gamma}$, implying $\tau(\gamma,r)\to+\infty$.  
To control the size of $\tau$, one may impose an upper bound $\gamma_{\bs\max}<4^{1-p}\widehat{\gamma}$ and choose $\gamma\le\gamma_{\bs\max}$.  
In this case,
    \[
    \tau\leq \left(\frac{(3\gamma_{\bs\max}  L_\nu+2)r^p+p\gamma_{\bs\max}(\gf(\ov{x})-c_1)}{2^{1-p}-c_0 p\gamma_{\bs\max}}\right)^{\frac{1}{p}}:=\widehat{\tau},
    \]
so that $\Tprox{\gf}{\gamma}{p}(x)\subseteq\mb(0;\widehat{\tau})$ for all $x\in\mb(0;r)$, and $\widehat{\tau}$ depends only on $r$.
\end{remark}

In the following theorem, we prove the H\"{o}lder property of HiFBE.
\begin{theorem}[H\"older property of HiFBE]\label{th:hol}
Let Assumptions~\ref{assum:mainassum:f}~and~\ref{assum:mainassum:g} hold, let $p=1+\nu$, and define $\gf:\R^n\to\Rinf$ by $\gf(x):=f(x)+g(x)$.
Then, for each $r>0$, there exists $\ov{\gamma}_r\in (0, \gamma^{g, p})$ 
such that, for every $\gamma\in(0,\ov\gamma_r)$, there exists a constant $\mathcal L_{\nu,r,\gamma}>0$ satisfying
\[
\vert \fgam{\gf}{p}{\gamma}(z)-\fgam{\gf}{p}{\gamma}(x)\vert\leq \mathcal L_{\nu,r,\gamma} \Vert z-x\Vert^\nu, \qquad \forall x, z\in \mb(0; r).
\]
In particular, $\fgam{\gf}{p}{\gamma}$ is H\"older continuous of order $\nu$ on $\mb(0;r)$ for every $\gamma\in(0,\ov\gamma_r)$.
\end{theorem}
\begin{proof}
Let $\ov{x}=0$ and $r>0$ be fixed. By Proposition~\ref{prop:findtau:lip}, there exists $\ov{\gamma}_0\in (0, \gamma^{g, p})$ such that
for each $\gamma\in \left(0, \ov{\gamma}_0\right)$ there is $\tau>0$ with
 $\Tprox{\gf}{\gamma}{p}(x)\subseteq \mb(\ov{x}; \tau)$  for all $x\in \mb(\ov{x}; r)$.
  By setting $\ov{\gamma}_r:=\min\{\ov{\gamma}_0, \frac{1}{L_\nu}\}$ and assuming $\gamma\in \left(0, \ov{\gamma}_r\right)$, we have
$\frac{L_\nu}{p}\leq \frac{1}{p\gamma}$ and for all $x, z\in \mb(\ov{x}; r)$ and $y\in \Tprox{\gf}{\gamma}{p}(x)$, we get
 \begin{align}\label{findtau:2:eq4:l1}
\fgam{\gf}{p}{\gamma}(z)&\leq f(z)+\langle \nabla f(z), y-z\rangle+g(y)+\frac{1}{p\gamma}\Vert z - y\Vert^p\nonumber\\
&\leq f(x)+\langle \nabla f(x), z - x\rangle+\frac{L_\nu}{p}\Vert z - x\Vert^p+\langle \nabla f(z), y-z\rangle+g(y)+\frac{1}{p\gamma}\Vert z - y\Vert^p
\nonumber\\
&\leq f(x)+\langle \nabla f(x), z - x\rangle+\frac{1}{p\gamma}\Vert z - x\Vert^p+\langle \nabla f(z), y-z\rangle+g(y)+\frac{1}{p\gamma}\Vert z - y\Vert^p.
 \end{align}
 Since
\[
\fgam{\gf}{p}{\gamma}(x)= f(x)+\langle \nabla f(x), y - x\rangle+g(y)+\frac{1}{p\gamma}\Vert x-y\Vert^p,
\]
then \eqref{findtau:2:eq4:l1} implies
 \begin{align}\label{findtau:2:eq4:l2}
\fgam{\gf}{p}{\gamma}(z)&
\leq \fgam{\gf}{p}{\gamma}(x)-\langle \nabla f(x), y - x\rangle-\frac{1}{p\gamma}\Vert x-y\Vert^p+\langle \nabla f(x), z - x\rangle+\frac{1}{p\gamma}\Vert z - x\Vert^p\nonumber\\&~~~+\langle \nabla f(z), y-z\rangle+\frac{1}{p\gamma}\Vert z - y\Vert^p
\nonumber\\
&=\fgam{\gf}{p}{\gamma}(x)+\langle \nabla f(x), z-y\rangle+\frac{1}{p\gamma}\left[\Vert z - x\Vert^p+\Vert z - y\Vert^p-\Vert x-y\Vert^p\right]
+\langle \nabla f(z), y-z\rangle\nonumber\\
&=\fgam{\gf}{p}{\gamma}(x)+\langle \nabla f(x)-\nabla f(z), z-y\rangle+\frac{1}{p\gamma}\left[\Vert z - x\Vert^p+\Vert z - y\Vert^p-\Vert x-y\Vert^p\right].
 \end{align}
Since $p>1$, the function $t\mapsto \frac{1}{p}t^p$ is convex on $[0, +\infty)$. Thus, the gradient inequality implies that for $t_1, t_2\in [0, +\infty)$,
$\frac{1}{p}t_1^p - \frac{1}{p}t_2^p\leq \langle t_1^{p-1}, t_1 - t_2\rangle$. Setting $t_1:=\Vert z - x\Vert +\Vert x -y\Vert$ and $t_2:=\Vert x- y\Vert$, we obtain
\[
 \frac{1}{p}\left(\Vert z - x\Vert+\Vert x - y\Vert\right)^p -\frac{1}{p}\Vert x - y\Vert^p\leq \left(\Vert z - x\Vert+\Vert x - y\Vert\right)^{p-1}\Vert z - x\Vert.
 \]
This consequently implies
  \begin{align}\label{findtau:2:eq4:l3}
\frac{1}{p\gamma}\left(\Vert z - y\Vert^p-\Vert x-y\Vert^p\right)&\leq \frac{1}{p\gamma} \left(\Vert z- x\Vert+\Vert x- y\Vert\right)^p-\frac{1}{p\gamma}\Vert x-y\Vert^p
\nonumber\\&\leq \frac{1}{\gamma}\left(\Vert z - x\Vert+\Vert x - y\Vert\right)^{p-1}\Vert z - x\Vert.
  \end{align}
Moreover, from $f\in \mathcal{C}^{1, \nu}_{L_\nu}(\mb(0; r))$ and the Cauchy–Schwarz inequality, we get
  \[
  \langle \nabla f(x)-\nabla f(z), z-y\rangle\leq \Vert \nabla f(x)-\nabla f(z)\Vert \Vert z-y\Vert \leq L_\nu\Vert x-z\Vert^{\nu}\Vert z-y\Vert.
  \]
  Together with \eqref{findtau:2:eq4:l2} and \eqref{findtau:2:eq4:l3}, this implies
  \begin{align}\label{findtau:2:eq4:l4}
\fgam{\gf}{p}{\gamma}(z)&\leq\fgam{\gf}{p}{\gamma}(x)+\langle \nabla f(x)-\nabla f(z), z-y\rangle+\frac{1}{p\gamma}\left[\Vert z - x\Vert^p+\Vert z - y\Vert^p-\Vert x-y\Vert^p\right]\nonumber
\\&\leq \fgam{\gf}{p}{\gamma}(x)+L_\nu\Vert z-x\Vert^{\nu} \Vert z-y\Vert+\frac{1}{p\gamma}\Vert z - x\Vert^p+ \frac{1}{\gamma}\left(\Vert z - x\Vert+\Vert x - y\Vert\right)^{p-1}\Vert z - x\Vert
\nonumber\\&\leq \fgam{\gf}{p}{\gamma}(x)+\left(L_\nu \left(\Vert z-x\Vert+\Vert x-y\Vert\right)+\frac{1}{p\gamma}\Vert z - x\Vert^{p-\nu}\right.
\nonumber\\&\qquad\qquad\qquad\qquad\left.+\frac{1}{\gamma}\left(\Vert z - x\Vert+\Vert x - y\Vert\right)^{p-1}\Vert z-x\Vert^{1-\nu}\right)\Vert z-x\Vert^{\nu}.
 \end{align}
Since $\Vert x\Vert<r$ and $\Vert y\Vert\leq \tau$, we have
$ \Vert x- y\Vert\leq r+\tau$.
In addition, since $\Vert x - z\Vert<2r$, from \eqref{findtau:2:eq4:l4},
\[
\fgam{\gf}{p}{\gamma}(z)\leq \fgam{\gf}{p}{\gamma}(x)+\left(L_\nu \left(3r+\tau\right)+\frac{(2r)^{p-\nu}}{p\gamma}+\frac{1}{\gamma}\left(3r+\tau\right)^{p-1}(2r)^{1-\nu}\right)\Vert z-x\Vert^{\nu}.
\]
Exchanging the roles of $x$ and $z$ yields the desired two-sided bound with
\[
\vert \fgam{\gf}{p}{\gamma}(z)-\fgam{\gf}{p}{\gamma}(x)\vert\leq L_{\nu,r,\gamma} \Vert z-x\Vert^{\nu}.
\]
where
$L_{\nu,r,\gamma}:=\left(L_\nu\left(3r+\tau\right)+\frac{(2r)^{p-\nu}}{p\gamma}+\frac{1}{\gamma}\left(3r+\tau\right)^{p-1}(2r)^{1-\nu}\right)$. 
In particular, this H\"older modulus depends on both $\gamma$ and $r$, giving our desired result.
\end{proof}
The H\"older constant in Theorem~\ref{th:hol} is local and is not uniform with respect to $\gamma$. Indeed, the quantity $\mathcal L_{\nu,r,\gamma}$ depends on the radius $r$ and on the parameter $\gamma$ through the bound $\tau(\gamma,r)$ obtained in Proposition~\ref{prop:findtau:lip}; see also Remark~\ref{rem:ontau}.

As a direct consequence of Theorem~\ref{th:hol}, we obtain Lipschitz continuity on a ball for the classical FBE ($p=2$).

\begin{corollary}[Lipschitz property of HiFBE]\label{cor:lip}
Let $p=2$ and Assumption~\ref{assum:mainassum:g} hold. Let
$f\in \mathcal{C}^{1, 1}_{L}(\R^n)$ and  $\gf(x):=f(x)+g(x)$. 
Then, for each $r>0$, there exists $\ov{\gamma}_r\in (0, \gamma^{g, p})$ such that, for every 
$\gamma\in \left(0, \ov{\gamma}_r\right)$, there exists a constant $\mathcal{L}_{r,\gamma}>0$ satisfying
\[
\vert \fgam{\gf}{p}{\gamma}(z)-\fgam{\gf}{p}{\gamma}(x)\vert\leq \mathcal{L}_{r,\gamma} \Vert z-x\Vert, \qquad \forall x, z\in \mb(0; r).
\]
\end{corollary}

Under the assumptions $f\in C^{1,1}(\mathbb{R}^n)$ and $g$ proper, lsc, and prox-bounded,  \cite[Proposition 4.2]{Themelis18} showed that the forward--backward envelope is real-valued and strictly continuous on $\mathbb{R}^n$ for every $\gamma\in(0,\gamma^{g,2})$. In finite dimensions, this implies local Lipschitz continuity. Therefore, the case $p=2$ in Corollary~\ref{cor:lip} is consistent with the known theory of the classical forward--backward envelope.

\subsection{{\bf Calmness and uniform boundedness}}
\label{subsec:calm}
In this subsection, we recall the notion of \textit{$p$-calmness}, which plays a key role in establishing the differentiability of HiFBE in the next section.

\begin{definition}[$p$-calm points]\label{def:calmp}\cite{Kabganidiff}
Let $p>1$ and $\gh: \R^n \to \Rinf$ be a proper lsc function, and let $\ov{x}\in \dom{\gh}$. 
Then, $\ov{x}$ is called a \textit{$p$-calm point} of $\gh$ with constant $M>0$ if
\[
\gh(x) + M\Vert x - \ov{x}\Vert^p > \gh(\ov{x}), \qquad \forall\, x\in \R^n,~x\neq \ov{x}.
\]
\end{definition}

The concept of $p$-calmness captures a localized sharpness of the objective. 
The following proposition highlights its relevance and shows how it connects minimizers, fixed points of $\Tprox{\gf}{\gamma}{p}$, and high-order prox-boundedness.

\begin{proposition}[Relationships with $p$-calmness]\label{lem:progpcalm}
Let $f:\R^n\to \R$ be Fr\'{e}chet differentiable, and $g:\R^n\to \Rinf$ be a proper lsc function. For $\gf(x):=f(x)+g(x)$, the following statements hold:
\begin{enumerate}[label=(\textbf{\alph*}), font=\normalfont\bfseries, leftmargin=0.7cm]
\item \label{lem:progpcalm:mincalm} If $\ov{x}\in \argmin{x\in \R^n}\gf(x)$, then $\ov{x}$ is a $p$-calm point of
$\gf$ for any $M>0$ and $p>1$;
\end{enumerate}
If $f\in \mathcal{C}^{1, \nu}_{L_\nu}(\R^n)$ and $p=1+\nu$, then, 
 \begin{enumerate}[label=(\textbf{\alph*}), font=\normalfont\bfseries, leftmargin=0.7cm, start=2]
\item \label{lem:progpcalm:critic2} 
if $\ov{x}$ is a $p$-calm point of $\gf$, then 
there exists $\ov{\gamma}>0$ such that for all 
$\gamma\in \left(0, \ov{\gamma}\right]$, we have  $\Tprox{\gf}{\gamma}{p}(\ov{x})=\{\ov{x}\}$;
\item \label{lem:progpcalm:critic3} if $\ov{x}$ is a $p$-calm point of $\gf$, then  $-\nabla f(\ov{x}) \in \partial g(\ov{x})$;
\item \label{lem:progpcalm:calm} if $\ov{x}$ is a $p$-calm point of $\gf$,  then $g$ is high-order prox-bounded.
\end{enumerate}
\end{proposition}
\begin{proof}
  \ref{lem:progpcalm:mincalm} It follows directly from the definition of a $p$-calm point.
  \\
\ref{lem:progpcalm:critic2} If $\ov{x}$ is a $p$-calm point of $\gf$ with constant $M>0$, then for each $\gamma\in \left(0, \bs\min\{\frac{1}{2L_\nu},\frac{1}{2pM}\}\right]$, we have
$L_\nu\leq \frac{1}{2\gamma}$ and $M\leq \frac{1}{2p\gamma}$. Thus, for all $x\in\R^n$ with $x\neq \ov{x}$,
\begin{align*}
 \gf(\ov{x})<\gf(x)+M\Vert x-\ov{x}\Vert^p&\leq \gf(x)+\frac{1}{2p\gamma}\Vert x-\ov{x}\Vert^p
 \\&\leq f(\ov{x})+\langle \nabla f(\ov{x}), x- \ov{x}\rangle+\frac{L_\nu}{p}\Vert x-\ov{x}\Vert^p+g(x)+\frac{1}{2p\gamma}\Vert x-\ov{x}\Vert^p,
 \\&\leq f(\ov{x})+\langle \nabla f(\ov{x}), x- \ov{x}\rangle+g(x)+\frac{1}{p\gamma}\Vert x-\ov{x}\Vert^p,
\end{align*}
which implies $\Tprox{\gf}{\gamma}{p}(\ov{x})=\{\ov{x}\}$.
Setting $\ov{\gamma}= \bs\min\{\frac{1}{2L_\nu},\frac{1}{2pM}\}$ completes the proof.
\\
\ref{lem:progpcalm:critic3} 
From Assertion~\ref{lem:progpcalm:critic2}, there exists $\ov{\gamma}>0$ such that for all 
$\gamma\in \left(0, \ov{\gamma}\right]$, we have  $\Tprox{\gf}{\gamma}{p}(\ov{x})=\{\ov{x}\}$.
Hence, by \cite[Exercise~8.8(c) and Theorem~10.1]{Rockafellar09}, 
for each $\gamma\in \left(0, \ov{\gamma}\right]$,
 \[
0\in \partial \left(f(\ov{x})+\langle \nabla f(\ov{x}) , \cdot - \ov{x}\rangle +g(\cdot)+\frac{1}{p\gamma}\Vert \ov{x}-\cdot\Vert^p\right)(\ov{x})
=\nabla f(\ov{x}) + \partial g(\ov{x}),
 \]
 which gives $-\nabla f(\ov{x}) \in \partial g(\ov{x})$.
\\
\ref{lem:progpcalm:calm} 
Let $\gamma>0$ be such that $M+\frac{L_\nu}{p}\le\frac{1}{p\gamma}$. 
Then, for all $y\in \R^n$,
\[
\begin{aligned}
\gf(\ov{x})&\leq \gf(y)+M\Vert y- \ov{x}\Vert^p
\\& \leq f(\ov{x})+\langle \nabla f(\ov{x}), y - \ov{x}\rangle + g(y)+\left(M+\frac{L_\nu}{p}\right)\Vert y- \ov{x}\Vert^p
\\& \leq f(\ov{x})+\langle \nabla f(\ov{x}), y - \ov{x}\rangle + g(y)+\frac{1}{p\gamma}\Vert y- \ov{x}\Vert^p.
\end{aligned}
\]
Hence, $\gf(\ov{x})\leq\fgam{\gf}{p}{\gamma}(\ov{x})$.
Together with  Fact~\ref{th:basichifbe}~\ref{th:basichifbe:finproxb}, $g$ is high-order prox-bounded.
\end{proof}

Proposition~\ref{prop:findtau:lip} shows that for any $r>0$ there exists a range of $\gamma$ such that $\Tprox{\gf}{\gamma}{p}\left(\mb(0; r)\right)$ is bounded. 
Next, we prove a complementary statement: for any target radius $\varepsilon>0$, there exists $r_\varepsilon>0$ such that
$\Tprox{\gf}{\gamma}{p}\left(\mb(0; r_\varepsilon)\right)\subseteq \mb(0; \varepsilon)$.

\begin{theorem}[Uniform boundedness of HiFBS]\label{th:prox:pcalm}
Let Assumption~\ref{assum:mainassum:f} hold with $p=1+\nu$, let $g:\R^n\to \Rinf$ be proper and lsc, and set $\gf(x):=f(x)+g(x)$.
Suppose $\ov{x}=0$ is a $p$-calm point of $\gf$ and $\gf(\ov{x})=0$. Then, there exists $\ov{\gamma}>0$ such that for every 
 $\gamma\in \left(0, \ov{\gamma}\right)$  and every $\varepsilon>0$ there is a neighborhood $U$ of $\ov{x}$ such that for any $x\in U$, $\Tprox{\gf}{\gamma}{p}(x)\neq \emptyset$. Moreover, if $y\in \Tprox{\gf}{\gamma}{p}(x)$, 
then,
\[
\Vert y\Vert< \varepsilon,\quad g(y)<g(\ov{x})+\varepsilon, \quad \frac{1}{\gamma}\Vert x-y\Vert^{p-1}+\Vert\nabla f(x) - \nabla f(\ov{x})\Vert< \varepsilon.
\]
\end{theorem}
\begin{proof}
By Proposition~\ref{lem:progpcalm}~\ref{lem:progpcalm:calm}, $g$ is high-order prox-bounded with threshold $\gamma^{g, p}>0$. Hence, by Fact~\ref{th:hopbfb}~\ref{th:hopbfb:proxb:proxnonemp}, $\Tprox{\gf}{\gamma}{p}(x)\neq \emptyset$ for all $x\in \R^n$ and $\gamma\in(0,\gamma^{g,p})$.
We now show the uniform local bounds on the elements of $\Tprox{\gf}{\gamma}{p}(x)$ near $\ov x$.

Let $M > 0$ be the $p$-calmness constant and choose 
\[\gamma\in \left(0, \ov{\gamma}:=\bs\min\left\{\frac{2^{-p}}{Mp}, \frac{1}{2L_\nu},\gamma^{g, p}\right\}\right).\]
For any $x\in\R^n$, it holds that
\[
\begin{aligned}
\fgam{\gf}{p}{\gamma}(x)&\leq  f(x)+\langle \nabla f(x), \ov{x}- x\rangle+g(\ov{x})+\frac{1}{p\gamma}\Vert x- \ov{x}\Vert^p
\\&\overset{(i)}{\leq} f(\ov{x})+\langle \nabla f(\ov{x}), x- \ov{x}\rangle+\frac{L_\nu}{p}\Vert x- \ov{x}\Vert^{p}+\langle \nabla f(x), \ov{x}- x\rangle+g(\ov{x})+\frac{1}{p\gamma}\Vert x- \ov{x}\Vert^p
\\&\overset{(ii)}{\leq} f(\ov{x})+\langle \nabla f(\ov{x})-\nabla f(x), x- \ov{x}\rangle+g(\ov{x})+\frac{2}{p\gamma}\Vert x- \ov{x}\Vert^p
\\&\overset{(iii)}{\leq}\frac{2L_\nu}{p} \Vert x\Vert^{p}+\frac{2}{p\gamma}\Vert x\Vert^p\overset{(iv)}{\leq}\frac{3}{p\gamma} \Vert x\Vert^{p},
\end{aligned}
\]
where $(i)$ uses $f\in \mathcal{C}^{1, \nu}_{L_{\nu}}(\R^n)$,
$(ii)$ uses $L_\nu<\frac{1}{2\gamma}<\frac{1}{\gamma}$,
 $(iii)$ uses Lemma~\ref{lem:hol:oneside} and $\gf(\ov{x})=0$, and $(iv)$ uses $2L_\nu<\frac{1}{\gamma}$. 
Thus,
\begin{equation}\label{eq1:lem:prox:pcalm}
\fgam{\gf}{p}{\gamma}(x)\leq \frac{3}{p\gamma}\Vert x\Vert^p.
\end{equation}
Fix an arbitrary $\delta>0$. By the definition of the infimum,
there exists $y\in\R^n$ such that
\begin{equation}\label{eq1b:lem:prox:pcalm}
f(x)+\langle \nabla f(x) , y - x\rangle +g(y)+\frac{1}{p\gamma}\Vert x-y\Vert^p\leq \fgam{\gf}{p}{\gamma}(x)+\delta.
\end{equation}
For such $y$, from \eqref{eq1:lem:prox:pcalm}, \eqref{eq1b:lem:prox:pcalm}, $p$-calmness of $\ov{x}$ with constant $M$, $\gf(\ov{x})=0$, and $\gamma<\frac{1}{2L_{\nu}}$, we have
\[
\begin{aligned}
-M\Vert y\Vert^{p}+\frac{1}{2p\gamma}\Vert x- y\Vert^p &\leq f(y)+g(y)+\frac{1}{2p\gamma}\Vert x- y\Vert^p
\nonumber\\&\leq f(x)+\langle \nabla f(x), y-x\rangle+\frac{L_\nu}{p}\Vert x- y\Vert^{p}+g(y)+\frac{1}{2p\gamma}\Vert x- y\Vert^{p}
\nonumber\\&\overset{(i)}{\leq}f(x)+\langle \nabla f(x), y-x\rangle+g(y)+\frac{1}{p\gamma}\Vert x- y\Vert^{p}
\nonumber\\& \leq \fgam{\gf}{p}{\gamma}(x)+\delta
\leq  \frac{3}{p\gamma}\Vert x\Vert^p+\delta,
\end{aligned}
\]
where $(i)$ is obtained from $\frac{L_\nu}{p}<\frac{1}{2p\gamma}$.
Hence, using \eqref{eq:baseq:pgen},
\[
 -M\Vert y\Vert^p+\frac{1}{2p\gamma}\left(2^{1-p}\Vert y\Vert^p-\Vert x\Vert^p\right)\leq  \frac{3}{p\gamma}\Vert x\Vert^p+\delta,
\]
i.e.,
$
(2^{1-p}-2Mp\gamma)\Vert y\Vert^p\leq 7\Vert x\Vert^p+2p\gamma\delta
$.
Since $\gamma<\frac{2^{-p}}{Mp}$, by setting $\mu:=(2^{1-p}-2Mp\gamma)^{-1}$, we have $\mu>0$ and
\begin{equation}\label{eq4:lem:prox:pcalm}
\Vert y\Vert\leq \left(7\mu\Vert x\Vert^p+2p\mu\gamma\delta\right)^{\frac{1}{p}}.
\end{equation}
In addition, from \eqref{eq1b:lem:prox:pcalm},
\[
\langle \nabla f(x), y -x \rangle +g(y)+\frac{1}{p\gamma}\Vert x- y\Vert^p\leq  \langle \nabla f(x), \ov{x} -x \rangle +g(\ov{x})+\frac{1}{p\gamma}\Vert x- \ov{x}\Vert^p+\delta.
\]
Hence,
\[
\begin{aligned}
g(y)+\frac{1}{p\gamma}\Vert x- y\Vert^p&\leq  \langle \nabla f(x), \ov{x} -x \rangle -\langle \nabla f(x), y -x \rangle +g(\ov{x})+\frac{1}{p\gamma}\Vert x- \ov{x}\Vert^p+\delta\\
&= \langle \nabla f(x), \ov{x} -y \rangle+g(\ov{x})+\frac{1}{p\gamma}\Vert x- \ov{x}\Vert^p+\delta\\
&\leq g(\ov{x})+\Vert \nabla f(x)\Vert \Vert y\Vert+\frac{1}{p\gamma}\Vert x\Vert^p+\delta\\
&\overset{(i)}{\leq} g(\ov{x})+\left(L_\nu\Vert x\Vert^\nu+ \Vert \nabla f(\ov{x})\Vert\right) \Vert y\Vert+\frac{1}{p\gamma}\Vert x\Vert^p+\delta\\
&\overset{(ii)}{\leq} g(\ov{x}) +\left(L_\nu\Vert x\Vert^\nu+ \Vert \nabla f(\ov{x})\Vert\right) \left(7\mu\Vert x\Vert^p+2p\mu\gamma\delta\right)^{\frac{1}{p}}+\frac{1}{p\gamma}\Vert x\Vert^p+\delta,
\end{aligned}
\]
where the inequality $(i)$ follows from
\[
\Vert \nabla f(x)\Vert-\Vert \nabla f(\ov{x})\Vert\leq \Vert \nabla f(x)- \nabla f(\ov{x})\Vert\leq L_\nu\Vert x\Vert^\nu,
\]
and $(ii)$ from \eqref{eq4:lem:prox:pcalm}.
Thus,
\begin{equation}\label{eq7:lem:prox:pcalm}
g(y)\leq g(\ov{x}) +\left(L_\nu\Vert x\Vert^\nu+ \Vert \nabla f(\ov{x})\Vert\right) \left(7\mu\Vert x\Vert^p+2p\mu\gamma\delta\right)^{\frac{1}{p}}+\frac{1}{p\gamma}\Vert x\Vert^p+\delta.
\end{equation}
Now, define $C:=\{y\in\R^n: \Vert y\Vert\leq \varepsilon\}$,
which is compact. Choose small $\delta>0$ and $\theta>0$ such that
\begin{align}
&\left(7\mu\theta^p+2p\mu\gamma\delta\right)^{\frac{1}{p}}< \varepsilon;\label{eq6a:lem:prox:pcalm}\\
& \frac{1}{\gamma}\left(\left(1+\left(7\mu\right)^{\frac{1}{p}}\right)\theta\right)^{p-1}+L_\nu\theta^\nu< \varepsilon;\label{eq6c:lem:prox:pcalm}\\
& \left(L_\nu\theta^\nu+ \Vert \nabla f(\ov{x})\Vert\right) \left(7\mu\theta^p+2p\mu\gamma\delta\right)^{\frac{1}{p}}+\frac{1}{p\gamma}\theta^p+\delta< \varepsilon.\label{eq6d:lem:prox:pcalm}
\end{align}
and define $U:=\{x\in\R^n: \Vert x\Vert< \theta\}$. Let $x\in U$ and $y$ satisfies \eqref{eq1b:lem:prox:pcalm}. Then, by \eqref{eq4:lem:prox:pcalm} and \eqref{eq6a:lem:prox:pcalm}, $\Vert y\Vert<\varepsilon$ and by \eqref{eq7:lem:prox:pcalm} and \eqref{eq6d:lem:prox:pcalm},
$g(y)<g(\ov{x})+\varepsilon$.
As such, if $x\in U$ and $y$ satisfies \eqref{eq1b:lem:prox:pcalm}, then $y\in C$. Specifically,
if $x\in U$, then 
\[\Tprox{\gf}{\gamma}{p}(x)=\argmin{y\in C} \left\{f(x)+\langle \nabla f(x), y-x\rangle+g(y)+\frac{1}{p\gamma}\Vert x- y\Vert^p\right\}.
\]
Assume that  $x\in U$ and $y\in \Tprox{\gf}{\gamma}{p}(x)$. Then, $y$ satisfies \eqref{eq1b:lem:prox:pcalm} for any $\delta>0$. Thus, by $\delta\downarrow 0$ and \eqref{eq4:lem:prox:pcalm},
\[
\Vert y\Vert\leq  \left(7\mu\Vert x\Vert^p\right)^{\frac{1}{p}}= \left(7\mu\right)^{\frac{1}{p}}\Vert x\Vert,
\]
i.e., 
\[
\Vert x - y\Vert\leq \Vert x\Vert +\left(7\mu\right)^{\frac{1}{p}}\Vert x\Vert=\left(1+\left(7\mu\right)^{\frac{1}{p}}\right)\Vert x\Vert.
\]
Moreover, by \eqref{eq6c:lem:prox:pcalm}, we get
\[
\begin{aligned}
 \frac{1}{\gamma}\Vert x-y\Vert^{p-1}+\Vert\nabla f(x) - \nabla f(\ov{x})\Vert&\leq \frac{1}{\gamma}\Vert x-y\Vert^{p-1}+L_\nu\Vert x\Vert^\nu
 \\&\leq  \frac{1}{\gamma}\left(\left(1+\left(7\mu\right)^{\frac{1}{p}}\right)\Vert x\Vert\right)^{p-1}+L_\nu\Vert x\Vert^\nu< \varepsilon,
\end{aligned}
\]
giving our desired result.

\end{proof}



\section{Differentiable properties and weak smoothness}\label{sec:diff}
In this section, we first present results concerning the Fr\'{e}chet/regular and Mordukhovich/limiting subdifferentials of HiFBE. Then, we obtain necessary and sufficient conditions for Fr\'{e}chet differentiability of HiFBE. Finally, we also provide some conditions for weak smoothness of HiFBE.

In \cite[Lemma~A.3]{kabgani2025fb}, a formula for the Mordukhovich subdifferential of HiFBE was established. The argument used there also suggests the expression that appears at the level of the Fr\'echet/regular subdifferential. The following lemma records this regular-subdifferential formula explicitly. Then, in Corollary~\ref{cor:morsubhifbe}, we derive the corresponding Mordukhovich subdifferential formula by passing from regular to limiting subgradients.

\begin{lemma}[Fr\'{e}chet/regular subdifferential of HiFBE]\label{lem:frech:hifbe}
Let $p>1$, Assumption~\ref{assum:mainassum:g} hold,
and let $\gf:\R^n\to\Rinf$ be defined by $\gf(x):=f(x)+g(x)$, where $f\in \mathcal{C}^2(U)$ for some open neighborhood $U$ of $\ov{x}\in \dom{\gf}$.
Then, for each $\gamma\in(0,\gamma^{g,p})$ and $\ov{y}\in \Tprox{\gf}{\gamma}{p}(\ov{x})$, we have
\begin{equation}\label{eq:lem:frech:hifbe:1}
\widehat{\partial} \fgam{\gf}{p}{\gamma}(\ov{x}) - \nabla^2 f(\ov{x})(\ov{y}-\ov{x}) 
\subseteq 
\left(\nabla f(\ov{x}) + \partial g(\ov{y})\right) 
\cap 
\left\{\frac{1}{\gamma}\Vert \ov{x}-\ov{y}\Vert^{p-2}(\ov{x}-\ov{y})\right\}.
\end{equation}
\end{lemma}
\begin{proof}
Let $\eta\in \widehat{\partial} \fgam{\gf}{p}{\gamma}(\ov{x})$ and $\ov{y}\in \Tprox{\gf}{\gamma}{p}(\ov{x})$ be arbitrary.
Define the function $\gh_1:\R^n\to\Rinf$ by
\[
\gh_1(x):=f(x)+\langle \nabla f(x), \ov{y}-\ov{x}\rangle + g(\ov{y}-\ov{x}+x).
\]
For any $x\in \R^n$, set $z_x := \ov{y}-\ov{x}+x$. By definition,
\begin{align*}
\fgam{\gf}{p}{\gamma}(x)
&\leq f(x) + \langle\nabla f(x), z_x - x\rangle + g(z_x) + \frac{1}{p\gamma}\Vert z_x - x\Vert^p \\
&= f(x) + \langle\nabla f(x), \ov{y}-\ov{x}\rangle + g(\ov{y}-\ov{x}+x) + \frac{1}{p\gamma}\Vert \ov{x}-\ov{y}\Vert^p.
\end{align*}
Hence, for any $x\in \R^n$,
\begin{align*}
 \fgam{\gf}{p}{\gamma}(x) - \fgam{\gf}{p}{\gamma}(\ov{x}) - \langle\eta, x - \ov{x}\rangle
 &\leq f(x) + \langle\nabla f(x), \ov{y}-\ov{x}\rangle + g(\ov{y}-\ov{x}+x) + \frac{1}{p\gamma}\Vert \ov{x}-\ov{y}\Vert^p \\
 &\quad - f(\ov{x}) - \langle\nabla f(\ov{x}), \ov{y}-\ov{x}\rangle - g(\ov{y}) - \frac{1}{p\gamma}\Vert \ov{x}-\ov{y}\Vert^p - \langle\eta, x - \ov{x}\rangle \\
 &= \gh_1(x) - \gh_1(\ov{x}) - \langle\eta, x - \ov{x}\rangle.
\end{align*}
Therefore,
\[
0 \leq 
{\mathop {\mathop {\bs\liminf }\limits_{x\to \ov{x},}}\limits_{x\neq \ov{x}}}
\frac{\fgam{\gf}{p}{\gamma}(x) - \fgam{\gf}{p}{\gamma}(\ov{x}) - \langle\eta, x - \ov{x}\rangle}{\Vert x - \ov{x}\Vert}
\leq 
{\mathop {\mathop {\bs\liminf }\limits_{x\to \ov{x},}}\limits_{x\neq \ov{x}}}
\frac{\gh_1(x) - \gh_1(\ov{x}) - \langle\eta, x - \ov{x}\rangle}{\Vert x - \ov{x}\Vert},
\]
which implies that $\eta\in \widehat{\partial} \gh_1(\ov{x})$.

By Fact~\ref{th:basichifbe}~\ref{th:basichifbe:finite}, $\fgam{\gf}{p}{\gamma}(\ov{x})$ is finite, and therefore $g(\ov{y})$ is finite as well. Consequently, the function $x\mapsto g(\ov{y}-\ov{x}+x)$ is finite at $\ov{x}$.
By \cite[Exercise 8.8 (c)]{Rockafellar09} and the inclusion $\widehat{\partial} g(\ov{y}) \subseteq \partial g(\ov{y})$, we obtain
\begin{equation}\label{eq:lem:frech:hifbe:2}
\eta\in \widehat{\partial} \gh_1(\ov{x}) = \nabla f(\ov{x}) + \nabla^2 f(\ov{x})(\ov{y}-\ov{x}) + \partial g(\ov{y}).
\end{equation}
Next, define $\gh_2:\R^n\to\R$ by $\gh_2(x):= f(x) + \langle \nabla f(x), \ov{y}-x\rangle + \frac{1}{p\gamma}\Vert \ov{y}-x\Vert^p$. Then,
\[
\begin{aligned}
 \fgam{\gf}{p}{\gamma}(x) - \fgam{\gf}{p}{\gamma}(\ov{x}) - \langle\eta, x - \ov{x}\rangle
 &\leq f(x) + \langle\nabla f(x), \ov{y}-x\rangle + g(\ov{y}) + \frac{1}{p\gamma}\Vert \ov{y}-x\Vert^p \\
 &\quad - f(\ov{x}) - \langle\nabla f(\ov{x}), \ov{y}-\ov{x}\rangle - g(\ov{y}) - \frac{1}{p\gamma}\Vert \ov{x}-\ov{y}\Vert^p - \langle\eta, x - \ov{x}\rangle \\
 &= \gh_2(x) - \gh_2(\ov{x}) - \langle\eta, x - \ov{x}\rangle.
\end{aligned}
\]
By a similar argument, this yields that $\eta$ is the gradient of $\gh_2$ at $\ov{x}$; i.e.,
\begin{equation}\label{eq:lem:frech:hifbe:3}
\eta = \nabla^2 f(\ov{x})(\ov{y}-\ov{x}) + \frac{1}{\gamma}\Vert \ov{x}-\ov{y}\Vert^{p-2}(\ov{x}-\ov{y}).
\end{equation}
Combining \eqref{eq:lem:frech:hifbe:2} and \eqref{eq:lem:frech:hifbe:3} completes the proof.
\end{proof}
As a corollary, we now discuss the Mordukhovich/limiting subdifferential of HiFBE.
\begin{corollary}[Mordukhovich/limiting subdifferential of HiFBE]\label{cor:morsubhifbe}
Let $p>1$, Assumption~\ref{assum:mainassum:g} hold,
and let $\gf:\R^n\to\Rinf$ be defined by $\gf(x):=f(x)+g(x)$, where $f\in \mathcal{C}^2(U)$ for some open neighborhood $U$ of $\ov{x}\in \dom{\gf}$.
Then, for each $\gamma\in(0,\gamma^{g,p})$, we have

\begin{equation}\label{eq:cor:morsubhifbe}
\partial\fgam{\gf}{p}{\gamma}(\ov{x})\subseteq \bigcup_{\ov{y}\in \Tprox{\gf}{\gamma}{p}(\ov{x})}\left\{\nabla^2 f(\ov{x})(\ov{y}-\ov{x})+\frac{1}{\gamma}\Vert \ov{x}-\ov{y}\Vert^{p-2}(\ov{x}-\ov{y})\right\}.
\end{equation}
\end{corollary}
\begin{proof}
Let $\ov{x}\in\R^n$ and $\ov{\zeta}\in \partial\fgam{\gf}{p}{\gamma}(\ov{x})$.  
By the definition of Mordukhovich (limiting) subdifferential, there exist sequences $x^k\to \ov{x}$ and $\{\zeta^k\}$ with $\zeta^k\in \widehat{\partial}\fgam{\gf}{p}{\gamma}(x^k)$ such that $\fgam{\gf}{p}{\gamma}(x^k)\to \fgam{\gf}{p}{\gamma}(\ov{x})$ and $\zeta^k\to \ov{\zeta}$.
For each $k$, by Lemma~\ref{lem:frech:hifbe}, we have
\[
\zeta^k = \nabla^2 f(x^k)(y^k - x^k) + \frac{1}{\gamma}\Vert x^k - y^k \Vert^{p-2}(x^k - y^k),
\]
where $y^k \in \Tprox{\gf}{\gamma}{p}(x^k)$.  
By Fact~\ref{th:hopbfb}~\ref{th:hopbfb:proxb:conv}, the sequence $\{y^k\}_{k\in\mathbb{N}}$ admits a cluster point $\ov{y}\in \Tprox{\gf}{\gamma}{p}(\ov{x})$ such that
\[
\ov{\zeta} = \nabla^2 f(\ov{x})(\ov{y}-\ov{x}) + \frac{1}{\gamma}\Vert \ov{x}-\ov{y}\Vert^{p-2}(\ov{x}-\ov{y}).
\]
Therefore, the inclusion \eqref{eq:cor:morsubhifbe} holds.
\end{proof}
\begin{remark}\label{rem:frech:hifbe}
From Lemma~\ref{lem:frech:hifbe}, the set $\widehat{\partial} \fgam{\gf}{p}{\gamma}(\ov{x})$ is either empty or a singleton. More precisely, if
$\widehat{\partial}\fgam{\gf}{p}{\gamma}(\ov{x})\neq\emptyset$, then for any
$\ov{y}\in \Tprox{\gf}{\gamma}{p}(\ov{x})$,
\[
\widehat{\partial} \fgam{\gf}{p}{\gamma}(\ov{x})
=
\left\{
\nabla^2 f(\ov{x})(\ov{y}-\ov{x})
+ \frac{1}{\gamma}\Vert \ov{x}-\ov{y}\Vert^{p-2}(\ov{x}-\ov{y})
\right\}.
\]
In general, the single-valuedness of $\widehat{\partial} \fgam{\gf}{p}{\gamma}(\ov{x})$ does not imply the differentiability of $\fgam{\gf}{p}{\gamma}$ at $\ov{x}$.
Nevertheless, if $\fgam{\gf}{p}{\gamma}$ is differentiable at $\ov{x}$, then for any
$\ov{y}\in \Tprox{\gf}{\gamma}{p}(\ov{x})$, we have
\[
\nabla \fgam{\gf}{p}{\gamma}(\ov{x})
=
\nabla^2 f(\ov{x})(\ov{y}-\ov{x})
+ \frac{1}{\gamma}\Vert \ov{x}-\ov{y}\Vert^{p-2}(\ov{x}-\ov{y}).
\]
For the particular case $p=2$, this can be expressed as
\[
\nabla \fgam{\gf}{p}{\gamma}(\ov{x})
=
\bs Q_\gamma(\ov{x}) \bs R_{\gamma,2}(\ov{x}),
\]
where $\bs Q_\gamma(\ov{x}) = \left(\gamma^{-1}\bs{{\rm Id}} - \nabla^2 f(\ov{x})\right)$.
This result is consistent with \cite[Theorem~2.6]{Stella17}.
\end{remark}

The following theorem establishes a necessary and sufficient condition for the differentiability of HiFBE.
\begin{theorem}[Characterization of differentiability of HiFBE]\label{th:char:diff:sing}
Let Assumptions~\ref{assum:mainassum:f}~and~\ref{assum:mainassum:g} hold, $p=1+\nu$, and for a given $r>0$,  $f\in \mathcal{C}^{2}(U)$
on an open set $U$ with $\bs{\rm cl}(\mb(0; r))\subseteq U$. Then the following assertions hold:
\begin{enumerate}[label=(\textbf{\alph*}), font=\normalfont\bfseries, leftmargin=0.7cm]
\item \label{th:char:diff:sing:a} 
There exists $\ov{\gamma}\in (0, \gamma^{g, p})$ such that for each 
 $\gamma\in \left(0, \ov{\gamma}\right)$, the condition $\fgam{\gf}{p}{\gamma}\in \mathcal{C}^1(\mb(0; r))$ implies that $\Tprox{\gf}{\gamma}{p}$ is single-valued and continuous on $\mb(0; r)$.
    
\item \label{th:char:diff:sing:b} If $\gamma\in (0, \gamma^{g, p})$ and $\Tprox{\gf}{\gamma}{p}$ is single-valued and continuous on $\mb(0; r)$, then  $\fgam{\gf}{p}{\gamma}\in \mathcal{C}^1(\mb(0; r))$ and for
$\ov{x}\in \mb(0; r)$ and
$\ov{y}=\Tprox{\gf}{\gamma}{p}(\ov{x})$,
\begin{equation}\label{eq:diffor:th:char:diff:sing}
\nabla \fgam{\gf}{p}{\gamma}(\ov{x})=\nabla^2f(\ov{x})(\ov{y}-\ov{x})+\frac{1}{\gamma}\Vert \ov{x}-\ov{y}\Vert^{p-2}(\ov{x}-\ov{y}).
\end{equation}
\end{enumerate}
\end{theorem}
\begin{proof}
\ref{th:char:diff:sing:a}
By Fact~\ref{th:hopbfb}~\ref{th:hopbfb:proxb:proxnonemp}, for each $\gamma\in (0, \gamma^{g, p})$ and $x\in \mb(0; r)$ we have $\Tprox{\gf}{\gamma}{p}(x)\neq \emptyset$.
Fix $\ov{x}\in \mb(0; r)$. From Lemma~\ref{lem:frech:hifbe} and Remark~\ref{rem:frech:hifbe}, for each  
$\ov{y}\in \Tprox{\gf}{\gamma}{p}(\ov{x})$,
\begin{equation}\label{eq:th:char:diff:sing:a1}
\nabla \fgam{\gf}{p}{\gamma}(\ov{x})=\nabla^2 f(\ov{x})(\ov{y}-\ov{x})
+\frac{1}{\gamma}\Vert \ov{x}-\ov{y}\Vert^{p-2}(\ov{x}-\ov{y}).
\end{equation}
We claim $\Tprox{\gf}{\gamma}{p}(\ov{x})$ is single-valued. Let $y_1, y_2\in \Tprox{\gf}{\gamma}{p}(\ov{x})$. Using \eqref{eq:th:char:diff:sing:a1}, we obtain
\[
\nabla^2f(\ov{x})(y_2-\ov{x})+\frac{1}{\gamma}\Vert \ov{x}-y_2\Vert^{p-2}(\ov{x}-y_2)=\nabla^2f(\ov{x})(y_1-\ov{x})+\frac{1}{\gamma}\Vert \ov{x}-y_1\Vert^{p-2}(\ov{x}-y_1).
\]
Thus
\begin{equation}\label{eq:th:char:diff:sing:a1-1}
\nabla^2f(\ov{x})(y_2-y_1)=\frac{1}{\gamma}\left[\Vert \ov{x}-y_1\Vert^{p-2}(\ov{x}-y_1)-\Vert \ov{x}-y_2\Vert^{p-2}(\ov{x}-y_2)\right].
\end{equation}
By Proposition~\ref{prop:findtau:lip} and Remark~\ref{rem:ontau}, choose $\widehat{\gamma}\in (0, \gamma^{g, p})$ and $\gamma_{\bs\max}\in(0,4^{1-p}\widehat{\gamma})$; for any $\gamma<\gamma_{\bs\max}$ there exists $\widehat{\tau}>0$, independent of $\gamma$, such that
$\Vert y_i\Vert\leq \widehat{\tau}$, $i=1,2$.
Hence $\Vert \ov{x}-y_i\Vert\leq r+\widehat{\tau}$, $i=1,2$.
Applying Lemma~\ref{lem:findlowbounknu:lemma}~\ref{lem:basicineq:a} and \eqref{eq:th:char:diff:sing:a1-1}, we have
\begin{align*}
\frac{1}{\gamma}\kappa_p (r+\widehat{\tau})^{p-2}\Vert y_2 - y_1\Vert&\leq \left\Vert \frac{1}{\gamma}\left[\Vert \ov{x}-y_1\Vert^{p-2}(\ov{x}-y_1)-\Vert \ov{x}-y_2\Vert^{p-2}(\ov{x}-y_2)\right]\right\Vert
\\&=\Vert \nabla^2f(\ov{x})(y_2-y_1)\Vert\leq\bs\max_{x\in\bs{\rm cl}(\mb(0; r))}\Vert \nabla^2f(x)\Vert \Vert y_2-y_1\Vert.
\end{align*}
Note that $L:=\bs\max_{x\in\bs{\rm cl}(\mb(0; r))}\Vert \nabla^2f(x)\Vert<\infty$.
If $L=0$, then immediately $y_1=y_2$; otherwise, for
\[
\gamma\in \Bigl(0,\;\ov{\gamma}:=\bs\min\Bigl\{\gamma_{\bs\max},\,\frac{\kappa_p (r+\widehat{\tau})^{p-2}}{L}\Bigr\}\Bigr),
\]
we also get $y_1=y_2$. Hence $\Tprox{\gf}{\gamma}{p}$ is single-valued on $\mb(0; r)$.

To prove continuity, let $x^k\to \ov{x}$. There exists some $K\in \mathbb{N}$ such that $x^k\in \mb(0; r)$ for all $k\ge K$.
Thus, for $k\geq K$, $\Tprox{\gf}{\gamma}{p}(x^k)$ is single-valued and by Fact~\ref{th:hopbfb}~\ref{th:hopbfb:proxb:conv}, the sequence 
$\{\Tprox{\gf}{\gamma}{p}(x^k)\}_{k\geq K}$ is bounded and its only cluster point is $\Tprox{\gf}{\gamma}{p}(\ov{x})$. Therefore $\Tprox{\gf}{\gamma}{p}(x^k)\to \Tprox{\gf}{\gamma}{p}(\ov{x})$, establishing continuity.
\\
\ref{th:char:diff:sing:b}
Assume $\Tprox{\gf}{\gamma}{p}$ is single-valued and continuous on $\mb(0; r)$.
Define
\[
F(x,y)\;:=\; f(x)+\langle \nabla f(x),y-x\rangle+\frac{1}{p\gamma}\|x-y\|^p+g(y).
\]
Then $\fgam{\gf}{p}{\gamma}(x)=F\left(x,\Tprox{\gf}{\gamma}{p}(x)\right)$, and for each fixed $y$, $F(\cdot,y)\in \mathcal{C}^{1}(\bs{\rm cl}(\mb(0; r)))$ with
\[
\nabla_x F(x,y)=\nabla^2 f(x)(y-x)+\frac{1}{\gamma}\|x-y\|^{p-2}(x-y).
\]
Fix $\ov{x}\in \mb(0; r)$ and $t\in \R^n$ with $\Vert t\Vert$ small enough so that
$\ov{x}+t\in \mb(0;r)$.
Using the definitions of $\Tprox{\gf}{\gamma}{p}(\ov{x}+t)$ and $\Tprox{\gf}{\gamma}{p}(\ov{x})$, we obtain
\begin{align*}
\fgam{\gf}{p}{\gamma}(\ov{x}+t)-\fgam{\gf}{p}{\gamma}(\ov{x})
&= F\left(\ov{x}+t,\Tprox{\gf}{\gamma}{p}(\ov{x}+t)\right)-F\left(\ov{x},\Tprox{\gf}{\gamma}{p}(\ov{x})\right) \\
&\leq F\left(\ov{x}+t,\Tprox{\gf}{\gamma}{p}(\ov{x})\right)-F\left(\ov{x},\Tprox{\gf}{\gamma}{p}(\ov{x})\right)
= \langle \nabla_x F(\ov{x},\Tprox{\gf}{\gamma}{p}(\ov{x})),t\rangle + o(\|t\|),    
\end{align*}
and
\begin{align*}
\fgam{\gf}{p}{\gamma}(\ov{x}+t)-\fgam{\gf}{p}{\gamma}(\ov{x})
\ge F\left(\ov{x}+t,\Tprox{\gf}{\gamma}{p}(\ov{x}+t)\right)-F\left(\ov{x},\Tprox{\gf}{\gamma}{p}(\ov{x}+t)\right)
= \langle \nabla_x F(\ov{x},\Tprox{\gf}{\gamma}{p}(\ov{x}+t)),t\rangle + o(\|t\|).
\end{align*}
By continuity of $\Tprox{\gf}{\gamma}{p}(\cdot)$ and of $y\mapsto\nabla_x F(\ov{x},\cdot)$, the last inner product equals
$\langle \nabla_x F(\ov{x},\Tprox{\gf}{\gamma}{p}(\ov{x})),t\rangle+o(\|t\|)$. Combining the two bounds yields
\[
\fgam{\gf}{p}{\gamma}(\ov{x}+t)-\fgam{\gf}{p}{\gamma}(\ov{x})-\langle \nabla_x F(\ov{x},\Tprox{\gf}{\gamma}{p}(\ov{x})),t\rangle = o(\|t\|),
\]
i.e., $\fgam{\gf}{p}{\gamma}$ is Fr\'{e}chet differentiable at $\ov{x}$ with
\[
\nabla \fgam{\gf}{p}{\gamma}(\ov{x})=\nabla_x F(\ov{x},\Tprox{\gf}{\gamma}{p}(\ov{x}))
=\nabla^2 f(\ov{x})(\Tprox{\gf}{\gamma}{p}(\ov{x})-\ov{x})+\frac{1}{\gamma}\|\ov{x}-\Tprox{\gf}{\gamma}{p}(\ov{x})\|^{p-2}(\ov{x}-\Tprox{\gf}{\gamma}{p}(\ov{x})).
\]
Since $\ov{x}$ was arbitrary and all ingredients are continuous on $\mb(0;r)$, we conclude $\fgam{\gf}{p}{\gamma}\in \mathcal{C}^1(\mb(0; r))$.
\end{proof}

In the upcoming results, we establish our results around  a $p$-calm point $\ov{x}$ of $\gf$ such that $g$ is prox-regular at $\ov{x}$ for 
$-\nabla f(\ov{x})\in\partial g(\ov{x})$. The assumption $-\nabla f(\ov{x})\in\partial g(\ov{x})$ is reasonable under $p$-calmness as discussed in 
Proposition~\ref{lem:progpcalm}~\ref{lem:progpcalm:critic3}. For simplicity, we assume that $\ov{x}=0$  and $\gf(\ov{x})=0$.

\begin{assumption}\label{assum:prox}
Let $\gf:\R^n\to\Rinf$ be defined by $\gf(x):=f(x)+g(x)$ where $f:\R^n\to \R$ is Fr\'{e}chet differentiable and $g:\R^n\to \Rinf$ is a proper lsc function. 
Let $\ov{x}=0$ be a $p$-calm point of $\gf$ and $\gf(\ov{x})=0$.
Let $g$ be prox-regular at $\ov{x}=0$ for $-\nabla f(\ov{x})\in\partial g(\ov{x})$.
\end{assumption}

The following theorem establishes the uniqueness, continuity, and H\"{o}lder continuity properties of HiFBS.
In view of the differentiability characterization of HiFBE in Theorem~\ref{th:char:diff:sing}, these results, in turn, provide sufficient conditions for the differentiability of HiFBE.

\begin{theorem}[Uniqueness and continuity of HiFBS under prox-regularity]\label{th:dif:cont}
Let Assumptions~\ref{assum:mainassum:f}~and~\ref{assum:prox} hold, and $p=1+\nu$.
Then, there exists $\ov{\gamma}>0$ such that for each 
 $\gamma\in \left(0, \ov{\gamma}\right)$,  there exists a neighborhood $U_\gamma$ of $\ov{x}$
 and a constant $\mathcal{L}_\nu>0$ such that 
$\Tprox{\gf}{\gamma}{p}$ is single-valued and continuous on $U_\gamma$. In addition, 
\begin{equation}\label{lochol:maineq:dif}
\Vert \Tprox{\gf}{\gamma}{p}(x_1) -  \Tprox{\gf}{\gamma}{p}(x_2) \Vert\leq \mathcal{L}_\nu  \Vert  x_1-x_2\Vert^{\frac{\nu}{2}}, \qquad \forall x_1, x_2\in U_\gamma.
\end{equation}
\end{theorem}
\begin{proof}
From Proposition~\ref{lem:progpcalm}~\ref{lem:progpcalm:calm}, $g$ is high-order prox-bounded with a threshold $\gamma^{g, p}>0$.
Let $M > 0$ be the constant of $p$-calmness. Since $g$ is prox-regular at $\ov{x}=0$ for $-\nabla f(\ov{x})\in \partial g(\ov{x})$, there exist $\varepsilon>0$ and $\rho>0$ such that
\begin{equation}\label{eq1ex1:th:dif:proxreg}
g(x')\geq g(x)+\langle \eta, x'-x\rangle-\frac{\rho}{2}\Vert x'-x\Vert^2,~\qquad \forall x'\in \mb(\ov{x};\varepsilon),
\end{equation}
whenever $\Vert x\Vert <\varepsilon$, $\eta\in \partial g(x)$, $\Vert \eta+\nabla f(\ov{x})\Vert< \varepsilon$, and $g(x)<g(\ov{x})+\varepsilon$. 

Let 
\[
\ov{\gamma}:=
\begin{cases}
\bs\min\left\{\dfrac{2^{-p}}{Mp},\,\dfrac1{2L_\nu},\,\gamma^{g,p}\right\}, & 1<p<2,\\[4ex]
\bs\min\left\{\dfrac{2^{-p}}{Mp},\,\dfrac1{2L_\nu},\,\gamma^{g,p},\,\dfrac{\kappa_p}{\rho}\right\}, & p=2.
\end{cases}
\]
Fix any $\gamma\in(0,\ov{\gamma})$.
By choosing any $\varepsilon'\in(0,\varepsilon]$ and replacing $\varepsilon$ by $\varepsilon'$, the function $g$ remains prox-regular with the same constant $\rho$.
Hence, without relabeling,
we may additionally assume that $\vartheta:=\kappa_p(2\varepsilon)^{p-2}>\rho\gamma$. Indeed, if $1<p<2$, this is ensured by choosing
$2\varepsilon<\left(\frac{\kappa_p}{\gamma\rho}\right)^{\frac{1}{2-p}}$,
while for $p=2$, since $\gamma<\kappa_p/\rho$, we have $\vartheta=\kappa_p>\rho\gamma$.
Here $\kappa_p>0$ depends only on $p$; see  Lemma \ref{lem:findlowbounknu:lemma}~\ref{lem:basicineq:a}.
By Theorem \ref{th:prox:pcalm}, there exists a neighborhood $U_\gamma\subseteq \mb(\ov{x};\varepsilon)$ such that for each $x\in U_\gamma$, we have $\Tprox{\gf}{\gamma}{p}(x)\neq \emptyset$. Additionally, for each $y\in \Tprox{\gf}{\gamma}{p}(x)$, we have
$\frac{1}{\gamma}\Vert x-y\Vert^{p-2}(x-y)-\nabla f(x)\in \partial g(y)$ and $\Vert y\Vert< \varepsilon$, $g(y)<g(\ov{x})+\varepsilon$, and
\[
\left\Vert\frac{1}{\gamma}\Vert x-y\Vert^{p-2}(x-y)-\nabla f(x) +\nabla f(\ov{x})\right\Vert\leq\frac{1}{\gamma}\Vert x-y\Vert^{p-1}+\Vert\nabla f(x)- \nabla f(\ov{x})\Vert< \varepsilon.
\]
Taking into account $x_i\in U_\gamma$ and $y_i\in \Tprox{\gf}{\gamma}{p}(x_i)$, $i=1,2$, and using \eqref{eq1ex1:th:dif:proxreg}, we obtain
\begin{equation}\label{eq1:th:dif:proxreg:dif}
g(y_2)\geq g(y_1)+\frac{1}{\gamma}\Vert x_1-y_1\Vert^{p-2}\langle x_1-y_1 , y_2-y_1\rangle-\langle \nabla f(x_1),  y_2-y_1\rangle-\frac{\rho}{2}\Vert y_2-y_1\Vert^2,
\end{equation}
and
\begin{equation}\label{eq2:th:dif:proxreg:dif}
g(y_1)\geq g(y_2)+\frac{1}{\gamma}\Vert x_2-y_2\Vert^{p-2}\langle x_2-y_2 , y_1-y_2\rangle-\langle \nabla f(x_2),  y_1-y_2\rangle-\frac{\rho}{2}\Vert y_1-y_2\Vert^2.
\end{equation}
Adding \eqref{eq1:th:dif:proxreg:dif} and \eqref{eq2:th:dif:proxreg:dif} results in
\begin{align}\label{eqn1::th:dif:proxreg:dif}
\frac{1}{\gamma} \langle \Vert x_2-y_2\Vert^{p-2} (x_2-y_2)- &\Vert x_1-y_1\Vert^{p-2} (x_1-y_1), y_2-y_1\rangle
\notag\\&\geq \langle \nabla f(x_1)-\nabla f(x_2),  y_1-y_2\rangle -\rho\Vert y_1-y_2\Vert^2.
\end{align}
On the other hand, since $\Vert x_i -y_i\Vert\leq 2\varepsilon$, $i=1,2$,  
Lemma~\ref{lem:findlowbounknu:lemma}~\ref{lem:basicineq:a} yields
 with $\vartheta=\kappa_p(2\varepsilon)^{p-2}$, 
\begin{align}\label{eqn2::th:dif:proxreg:dif}
\langle \Vert x_2-y_2\Vert^{p-2}(x_2-y_2) - \Vert x_1-y_1\Vert^{p-2}(x_1-y_1), (x_2-y_2)-(x_1-y_1)\rangle&\geq  \vartheta\Vert (x_2-y_2) - (x_1-y_1)\Vert^{2}\notag\\
&\geq \vartheta\big\vert\Vert y_2-y_1\Vert -\Vert x_2-x_1\Vert\big\vert^{2}.
\end{align}
Setting $a=\Vert y_2-y_1\Vert$ and $b=\Vert x_2-x_1\Vert$ and using the inequalities
$\vert a -b\vert^{2} -a^{2}\geq -2\vert a\vert \vert b\vert$ and $\Vert y_2-y_1\Vert\leq 2\varepsilon$, we have
$\big\vert\Vert y_2-y_1\Vert -\Vert x_2-x_1\Vert\big\vert^{2}\geq -4\varepsilon \Vert x_2-x_1\Vert+\Vert y_2-y_1\Vert^{2}$. Together with \eqref{eqn1::th:dif:proxreg:dif} and \eqref{eqn2::th:dif:proxreg:dif}, this implies
\begin{align}\label{eqn3::th:dif:proxreg:dif}
\vartheta\left(4\varepsilon \Vert x_2-x_1\Vert-\Vert y_2-y_1\Vert^{2}\right)  &\geq -\vartheta\big\vert\Vert y_2-y_1\Vert -\Vert x_2-x_1\Vert\big\vert^{2} 
\nonumber\\&\geq
\langle \Vert x_2-y_2\Vert^{p-2}(x_2-y_2) - \Vert x_1-y_1\Vert^{p-2}(x_1-y_1), (x_1-y_1)-(x_2-y_2)\rangle
\nonumber\\
&=\langle \Vert x_2-y_2\Vert^{p-2}(x_2-y_2) - \Vert x_1-y_1\Vert^{p-2}(x_1-y_1), x_1-x_2\rangle\nonumber\\
&~~~+\langle \Vert x_2-y_2\Vert^{p-2}(x_2-y_2) - \Vert x_1-y_1\Vert^{p-2}(x_1-y_1), y_2-y_1\rangle\nonumber\\
&\geq \langle \Vert x_2-y_2\Vert^{p-2}(x_2-y_2) - \Vert x_1-y_1\Vert^{p-2}(x_1-y_1), x_1-x_2\rangle\nonumber\\
&~~~+\gamma\langle\nabla f(x_1)-\nabla f(x_2),  y_1-y_2\rangle -\rho\gamma\Vert y_2 - y_1\Vert^{2}
\nonumber\\
&\geq \langle \Vert x_2-y_2\Vert^{p-2}(x_2-y_2) - \Vert x_1-y_1\Vert^{p-2}(x_1-y_1), x_1-x_2\rangle\nonumber\\
&~~~-\gamma L_\nu\Vert x_1-x_2\Vert^\nu \Vert y_1-y_2\Vert -\rho\gamma\Vert y_2 - y_1\Vert^{2}
\nonumber\\
&\geq -\left((\Vert x_2-y_2\Vert^{p-1}+\Vert x_1-y_1\Vert^{p-1})\Vert x_ 2-x_1\Vert^{1-\nu}+\gamma L_\nu \Vert y_1-y_2\Vert \right)\Vert x_ 2-x_1\Vert^\nu
\nonumber\\&~~~-\rho\gamma\Vert y_2 - y_1\Vert^{2}\nonumber\\
&\overset{(i)}{\geq} -(2+\gamma L_\nu)(2\varepsilon)\Vert x_2-x_1\Vert^\nu-\rho\gamma\Vert y_2 - y_1\Vert^{2},
\end{align}
where $(i)$ follows from $p-1=\nu$, and
\[
(\Vert x_2-y_2\Vert^{p-1}+\Vert x_1-y_1\Vert^{p-1})\Vert x_ 2-x_1\Vert^{1-\nu}+\gamma L_\nu \Vert y_1-y_2\Vert \leq
(2(2\varepsilon)^\nu)(2\varepsilon)^{1-\nu}+\gamma L_\nu(2\varepsilon).
\]
From \eqref{eqn3::th:dif:proxreg:dif}, we obtain
\[
(\vartheta-\rho\gamma)\Vert y_2 - y_1\Vert^{2}\leq  (2\vartheta(2\varepsilon)^{1-\nu}+(2+\gamma L_\nu))(2\varepsilon) \Vert x_2-x_1\Vert^\nu.
\]
Since $\vartheta-\rho\gamma>0$, this ensures
\begin{equation}\label{lochol:eq11:dif}
\Vert y_2 - y_1\Vert\leq \mathcal{L}_\nu   \Vert x_1-x_2\Vert^{\frac{\nu}{2}},
\end{equation}
with $\mathcal{L}_\nu :=\left(\frac{(2\vartheta(2\varepsilon)^{1-\nu}+(2+\gamma L_\nu))(2\varepsilon)}{\vartheta-\rho\gamma}\right)^{\frac{1}{2}}$.
Equation \eqref{lochol:eq11:dif} yields the single-valuedness of $\Tprox{\gf}{\gamma}{p}(x)$
for any $x\in U_\gamma$ and proves \eqref{lochol:maineq:dif}. It also shows the continuity of
$\Tprox{\gf}{\gamma}{p}$ on $U_\gamma$.  
\end{proof}

In the following theorem, we use the notation $\mathcal{C}^{2, \mu}_{L_\mu}(U)$, referring to the class of twice differentiable functions on $U\subseteq\R^n$ whose second derivative is $\mu$-H\"{o}lder continuous on $U$.
\begin{theorem}[Differentiability and weak smoothness of HiFBE]\label{th:weaksm}
Let Assumptions~\ref{assum:mainassum:f}~and~\ref{assum:prox} hold, and $p=1+\nu$.
If $f\in \mathcal{C}^{2}(U)$ on an open neighborhood $U$ of $\ov{x}$, then the following statements hold.
\begin{enumerate}[label=(\textbf{\alph*}), font=\normalfont\bfseries, leftmargin=0.7cm]
\item\label{th:weaksm:a}  there exists $\ov{\gamma}>0$ such that for each 
 $\gamma\in \left(0, \ov{\gamma}\right)$,  there exists $r>0$ such that 
$\fgam{\gf}{p}{\gamma}\in \mathcal{C}^1(\mb(0; r))$.
\end{enumerate}
Moreover, if $f\in \mathcal{C}^{2, \mu}_{L_\mu}(U)$ for some $L_\mu>0$ with $\mu\in (0,1)$, then
\begin{enumerate}[label=(\textbf{\alph*}), font=\normalfont\bfseries, leftmargin=0.7cm, start=2]
 \item\label{th:weaksm:b} by setting $\eta=\bs\min\{\mu,\nu^2/2\}$, there exists $L_\eta>0$ such that for each $x_1, x_2\in \mb(0; r)$,
\[
\Vert\nabla \fgam{\gf}{p}{\gamma}(x_2)- \nabla \fgam{\gf}{p}{\gamma}(x_1)\Vert\leq
L_\eta\Vert x_2-x_1\Vert^\eta.
\]
\end{enumerate}
\end{theorem}
\begin{proof}
\ref{th:weaksm:a} From Proposition~\ref{lem:progpcalm}~\ref{lem:progpcalm:calm}, $g$ is high-order prox-bounded with a threshold $\gamma^{g, p}>0$.
By Theorem~\ref{th:dif:cont}, there exists $\ov{\gamma}\in(0,\gamma^{g,p})$ such that for each 
 $\gamma\in \left(0, \ov{\gamma}\right)$,  there exists a neighborhood $U_\gamma$ of $\ov{x}$ such that 
$\Tprox{\gf}{\gamma}{p}$ is single-valued and continuous on $U_\gamma$. Fix $\gamma\in \left(0, \ov{\gamma}\right)$ and choose $r>0$ such that
$\bs{\rm cl}(\mb(0; r))\subseteq U\cap U_\gamma$. The high-order prox-boundedness of $g$ and 
Theorem~\ref{th:char:diff:sing} imply that $\fgam{\gf}{p}{\gamma}\in \mathcal{C}^1(\mb(0; r))$.
\\
\ref{th:weaksm:b} Let $x_1, x_2\in \mb(0; r)$  and $y_i=\Tprox{\gf}{\gamma}{p}(x_i)$,  $i=1, 2$. 
From \eqref{eq:diffor:th:char:diff:sing}, 
\begin{align}\label{eq1:th:weaksm}
\Vert\nabla \fgam{\gf}{p}{\gamma}(x_2)- \nabla \fgam{\gf}{p}{\gamma}(x_1)\Vert&\leq 
\Vert \nabla^2f(x_2)(y_2-x_2) - \nabla^2f(x_1)(y_1-x_1)\Vert
\notag\\&~~~~+
\frac{1}{\gamma}\left\Vert\Vert x_2-y_2\Vert^{p-2}(x_2-y_2)-\Vert x_1-y_1\Vert^{p-2}(x_1-y_1)\Vert\right\Vert.
\end{align}
Since $\Tprox{\gf}{\gamma}{p}$ is continuous on the compact set $\closure{\mb(0; r)}$, there exists $C_r>0$ such that
\[
\|\Tprox{\gf}{\gamma}{p}(x)\|\leq C_r,\qquad \forall x\in \mb(0; r).
\]
Hence, for $i=1,2$,
\[
\|y_i-x_i\|\leq \|y_i\|+\|x_i\|\leq C_r+r=:D_r.
\]
Moreover, since $f\in C^{2,\mu}_{L_\mu}(U)$ and $\closure{\mb(0; r)}\subseteq U$, we may define
\[
M_r:={\mathop {\bs\sup}\limits_{x\in \closure{\mb(0; r)}}}\|\nabla^2f(x)\|<+\infty.
\]
Writing
\[
\nabla^2f(x_2)(y_2-x_2)-\nabla^2f(x_1)(y_1-x_1) = \big(\nabla^2f(x_2)-\nabla^2f(x_1)\big)(y_2-x_2)+\nabla^2f(x_1)\big((y_2-x_2)-(y_1-x_1)\big),
\]
and setting $\alpha:=\min\{\mu,\nu/2\}$,
the $\mu$-H\"older continuity of $\nabla^2 f$ and \eqref{lochol:maineq:dif} imply that
\begin{align}\label{eq2:th:weaksm}
\Vert \nabla^2f(x_2)(y_2-x_2) - \nabla^2f(x_1)(y_1-x_1)\Vert&\leq L_\mu\Vert x_2-x_1\Vert^\mu D_r
+M_r\left(\Vert x_2-x_1\Vert+\Vert y_2-y_1\Vert\right)
\notag\\&\leq L_\mu D_r\Vert x_2-x_1\Vert^\mu+M_r\left(\Vert x_2-x_1\Vert+\mathcal{L}_\nu  \Vert  x_2-x_1\Vert^{\frac{\nu}{2}}\right)
\notag\\&\leq \left(L_\mu D_r(2r)^{\mu-\alpha}+M_r\left((2r)^{1-\alpha}+\mathcal{L}_\nu  (2r)^{\frac{\nu}{2}-\alpha}\right)\right)\Vert x_2-x_1\Vert^\alpha
\notag\\&=: A_r \|x_2-x_1\|^\alpha.
\end{align}
Moreover,
  \begin{align*}
\frac{1}{\gamma}\left\Vert\Vert x_2-y_2\Vert^{p-2}(x_2-y_2)-\Vert x_1-y_1\Vert^{p-2}(x_1-y_1)\Vert\right\Vert &=\left\Vert \nabla\left(\frac{1}{p\gamma}\Vert\cdot\Vert^{p}\right)\left(x_2 - y_2\right)-\nabla\left(\frac{1}{p\gamma}\Vert\cdot\Vert^{p}\right)\left(x_1 -y_1\right)\right\Vert\\
 &\leq \frac{2^{2-p}}{\gamma} \left\Vert (x_2-x_1) - (y_2-y_1)\right\Vert^{p-1},
  \end{align*}
  where the last inequality comes from \cite[Theorem~6.3]{Rodomanov2020}. Together with \eqref{lochol:maineq:dif}, this ensures
  \begin{align}\label{eq3:th:weaksm}
\frac{1}{\gamma}\left\Vert\Vert x_2-y_2\Vert^{p-2}(x_2-y_2)-\Vert x_1-y_1\Vert^{p-2}(x_1-y_1)\Vert\right\Vert  &\leq \frac{2^{2-p}}{\gamma}\left(\Vert x_2-x_1\Vert + \Vert y_2-y_1\Vert\right)^{\nu}\notag\\
 &\leq  \frac{2^{2-p}}{\gamma}\left(\Vert x_2-x_1\Vert^{1-\frac{\nu}{2}}\Vert x_2-x_1\Vert^{\frac{\nu}{2}} +\mathcal{L}_\nu  \Vert  x_1-x_2\Vert^{\frac{\nu}{2}}\right)^{\nu}\notag\\
&\leq \frac{2^{2-p}}{\gamma}\left((2r)^{1-\frac{\nu}{2}}+\mathcal{L}_\nu \right)^{\nu} \Vert x_2-x_1\Vert^{\frac{\nu^2}{2}}\notag\\
&=: B_r \Vert x_2-x_1\Vert^{\frac{\nu^2}{2}}
  \end{align}
Combining \eqref{eq1:th:weaksm}, \eqref{eq2:th:weaksm}, and \eqref{eq3:th:weaksm}, and setting
$\eta:=\bs\min\{\alpha,\nu^2/2\}=\bs\min\{\mu,\nu^2/2\}$, there exists $L_\eta>0$ such that 
\[
\Vert\nabla \fgam{\gf}{p}{\gamma}(x_2)- \nabla \fgam{\gf}{p}{\gamma}(x_1)\Vert\leq
L_\eta\Vert x_2-x_1\Vert^\eta.
\]
This completes the proof.
\end{proof}

To illustrate the role of $p$-calmness in Theorems~\ref{th:weaksm}~and~\ref{th:dif:cont}, the following example presents a function with both a non-$p$-calm point and $p$-calm points, including one that is not a global minimizer, and shows that the HiFBE can fail to be differentiable at a non-$p$-calm point.

\begin{example}[An illustration of $p$-calmness]\label{ex:p-calm-vs-non-p-calm}
Let $p=\frac{3}{2}$, $f\equiv0$, and define $g:\R\to\R$ by $g(x):=-\vert x\vert+s(x)$ where
\[
s(x):=x^2-4e^{-(x-2)^2}-4e^{-(x+2)^2}+8e^{-4}-\frac12\left(e^{-((x-0.8)/0.18)^2}+e^{-((x+0.8)/0.18)^2}-2e^{-(0.8/0.18)^2}\right).
\]
Set $\gf(x):=f(x)+g(x)=g(x)=-|x|+s(x)$. 
We proceed in four steps. First, we show that $\ov{x}=0$ is not a $p$-calm point. Next, we identify two $p$-calm points of $\gf$, namely a non-global local minimizer $x_{\mathrm{loc}}$ and a global minimizer $x^\star$. We then verify, by Theorems~\ref{th:dif:cont}~and~\ref{th:weaksm}, that the envelope $\fgam{\gf}{p}{\gamma}$ is differentiable in neighborhoods of these $p$-calm points for sufficiently small $\gamma>0$. Finally, we prove that $\fgam{\gf}{p}{\gamma}$ fails to be differentiable at the non-$p$-calm point $\ov{x}=0$ for every $\gamma>0$.

We first show that $\ov{x}=0$ is not $p$-calm.
We have $s\in \mathcal{C}^\infty(\R)$, $s(0)=0$, and, by symmetry, $s'(0)=0$. Therefore, the Taylor expansion of $s$ at $\ov{x}=0$ starts at order at least two, and hence $s(x)=\mathcal{O}(x^2)$.
Thus,
\[
\gf(x)=\gf(0)-|x|+\mathcal{O}(x^2),
\]
since $\gf(0)=0$.
Let $M>0$ be arbitrary. From the above expansion, there exists $\delta_1>0$ such that
\begin{equation}\label{eq1:ex:p-calm-vs-non-p-calm}
    \gf(x)-\gf(0)\le -\frac12|x|,\qquad \forall\, x\in(-\delta_1,\delta_1).
\end{equation}
Moreover, there exists $\delta_2>0$ such that
\begin{equation}\label{eq2:ex:p-calm-vs-non-p-calm}
M|x|^{3/2}\le \frac14|x|,\qquad \forall\, x\in(-\delta_2,\delta_2).
\end{equation}
Hence, for every $x\neq 0$ with $|x|<\min\{\delta_1,\delta_2\}$,
\[
\gf(x)+M|x|^{3/2}\le\gf(0)-\frac12|x|+\frac14|x|<\gf(0).
\]
Since this happens for every $M>0$, the point $\ov{x}=0$ is not a $p$-calm point of $\gf$ in the
sense of Definition~\ref{def:calmp}.
A numerical inspection shows that $\gf$ has two positive strict local minimizers, $x_{\mathrm{loc}}\approx 0.86$ and $x^\star\approx 1.67$
with $\gf(x_{\mathrm{loc}})\approx -1.51$ and $\gf(x^\star)\approx -2.32$. Hence $x_{\mathrm{loc}}$ is a strict local minimizer but not a global one, whereas $x^\star$ is a global minimizer.
Since $x^\star\in\argmin{x\in\R} \gf(x)$, Proposition~\ref{lem:progpcalm}~\ref{lem:progpcalm:mincalm} implies that $x^\star$ is a $p$-calm point. We next show that the non-global local minimizer $x_{\mathrm{loc}}$ is also $p$-calm.

It follows from $x_{\mathrm{loc}}\neq 0$ that the absolute-value term is smooth in a neighborhood of $x_{\mathrm{loc}}$, and
therefore $\gf$ is of class $\mathcal{C}^\infty$ near $x_{\mathrm{loc}}$. Since $x_{\mathrm{loc}}$ is a strict local minimizer,
there exists $\varepsilon>0$ such that
\[
\gf(x)>\gf(x_{\mathrm{loc}}),\qquad\forall\, x\neq x_{\mathrm{loc}}\ \text{with}\ |x-x_{\mathrm{loc}}|<\varepsilon.
\]
Define
\[
h(x):=\frac{\gf(x_{\mathrm{loc}})-\gf(x)}{|x-x_{\mathrm{loc}}|^{3/2}},\qquad x\neq x_{\mathrm{loc}}.
\]
Then $h(x)<0$ whenever $0<|x-x_{\mathrm{loc}}|<\varepsilon$. On the other hand, as $|x|\to\infty$, the Gaussian terms vanish, and hence
\begin{equation}\label{eq3:ex:p-calm-vs-non-p-calm}
\gf(x)=x^2-|x|+\mathcal{O}(1),
\end{equation}
i.e.,
\[
h(x)=\frac{\gf(x_{\mathrm{loc}})-\gf(x)}{|x-x_{\mathrm{loc}}|^{3/2}}=\frac{\gf(x_{\mathrm{loc}})-x^2+|x|+\mathcal{O}(1)}{|x-x_{\mathrm{loc}}|^{3/2}}
\to -\infty
\qquad (|x|\to\infty).
\]
Hence, for sufficiently large $R>0$, the function $h$ attains a finite maximum on the compact set
\[
K:=\left\{x\in\R:\ \varepsilon\le |x-x_{\mathrm{loc}}|\le R\right\},
\]
i.e.,
\[
M_{\mathrm{loc}}:=\sup_{x\neq x_{\mathrm{loc}}}\frac{\gf(x_{\mathrm{loc}})-\gf(x)}{|x-x_{\mathrm{loc}}|^{3/2}}<+\infty.
\]
Choosing any $M>M_{\mathrm{loc}}$, we obtain
\[
\gf(x)+M|x-x_{\mathrm{loc}}|^{3/2}>\gf(x_{\mathrm{loc}}),\qquad \forall\, x\neq x_{\mathrm{loc}},
\]
which proves that $x_{\mathrm{loc}}$ is a $p$-calm point of $\gf$, although it is not a global minimizer.

Since both $x_{\mathrm{loc}}$ and $x^\star$ are nonzero, the function $g=\gf$ is $\mathcal{C}^\infty$ in neighborhoods of these points and is therefore prox-regular there. Moreover, $f\equiv 0$ belongs to
$\mathcal{C}^{1,\nu}_{L_\nu}(\R)$ for every $\nu\in(0,1]$, in particular for $\nu=\frac12$, so that
$p=1+\nu=\frac32$.
Hence, after a local translation and normalization, the assumptions of Theorems~\ref{th:dif:cont} and~\ref{th:weaksm} are satisfied at both $x_{\mathrm{loc}}$ and $x^\star$.
Therefore, for sufficiently small $\gamma>0$,
the envelope $\fgam{\gf}{p}{\gamma}$ is differentiable in neighborhoods of $x_{\mathrm{loc}}$ and $x^\star$.

Now, we show the non-differentiability of $\fgam{\gf}{p}{\gamma}$  at $\ov x=0$ for every $\gamma>0$. Let $\gamma>0$ be arbitrary and define $m_\gamma:\R\to\R$ as
\[
m_\gamma(x):=\gf(x)+\frac{2}{3\gamma}|x|^{\frac{3}{2}}.
\]
Since $\gf$ is even, the function $m_\gamma$ is also even. Moreover, there exist $\delta_1, \delta_2>0$ such that $\gf$ satisfies \eqref{eq1:ex:p-calm-vs-non-p-calm} and \eqref{eq2:ex:p-calm-vs-non-p-calm} with 
$M=\frac{2}{3\gamma}$. Hence, for every $x\neq 0$ with $|x|<\min\{\delta_1,\delta_2\}$,
\[
m_\gamma(x)=\gf(x)+\frac{2}{3\gamma}|x|^{3/2}\le\gf(0)-\frac{1}{2}|x|+\frac{1}{4}|x|<\gf(0)=m_\gamma(0).
\]
Thus, $\ov x =0$ is not a minimizer of $m_\gamma$.
On the other hand, as $|x|\to\infty$, \eqref{eq3:ex:p-calm-vs-non-p-calm} holds. Thus, 
$m_\gamma(x)\to +\infty$ as $|x|\to\infty$, which means that $m_\gamma$ is a coercive function. Hence $m_\gamma$ attains a global minimum at some point $x_\gamma\neq 0$. Since $m_\gamma$ is even,
$-x_\gamma$ is also a minimizer. Therefore, $\Tprox{\gf}{\gamma}{p}(0)$ contains at least two distinct points for every $\gamma>0$. Assume now, by contradiction, that  $\fgam{\gf}{p}{\gamma}$ is differentiable at $\ov{x}= 0$. Then, by Remark~\ref{rem:frech:hifbe},
for every $y\in \Tprox{\gf}{\gamma}{p}(0)$,
\[
\nabla\fgam{\gf}{p}{\gamma}(0)=\frac{1}{\gamma}|0-y|^{p-2}(0-y)
= -\frac{1}{\gamma}\operatorname{sign}(y)\,|y|^{1/2}.
\]
Evaluating this identity at the two distinct minimizers $x_\gamma$ and $-x_\gamma$ yields two different values, which is impossible. Hence $\fgam{\gf}{p}{\gamma}$ is not differentiable at $x=0$ for every $\gamma>0$.
In particular, for every sufficiently small $\gamma>0$ in the range of Theorem~\ref{th:char:diff:sing}~\ref{th:char:diff:sing:a}, the multivaluedness of $\Tprox{\gf}{\gamma}{p}(0)$  also rules out the possibility that $\fgam{\gf}{p}{\gamma}$ belongs to $\mathcal{C}^1$ on any neighborhood of $\ov{x}=0$.

This behavior is illustrated in Figure~\ref{fig:calm-noncalm}, where the graph of $\gf$ and its
envelope $\fgam{\gf}{p}{\gamma}$ highlights the non-$p$-calm point at $\ov{x}=0$ together with the two $p$-calm points $x_{\mathrm{loc}}$ and $x^\star$.
\end{example}
\begin{figure}[t]
\centering
\includegraphics[width=0.78\textwidth]{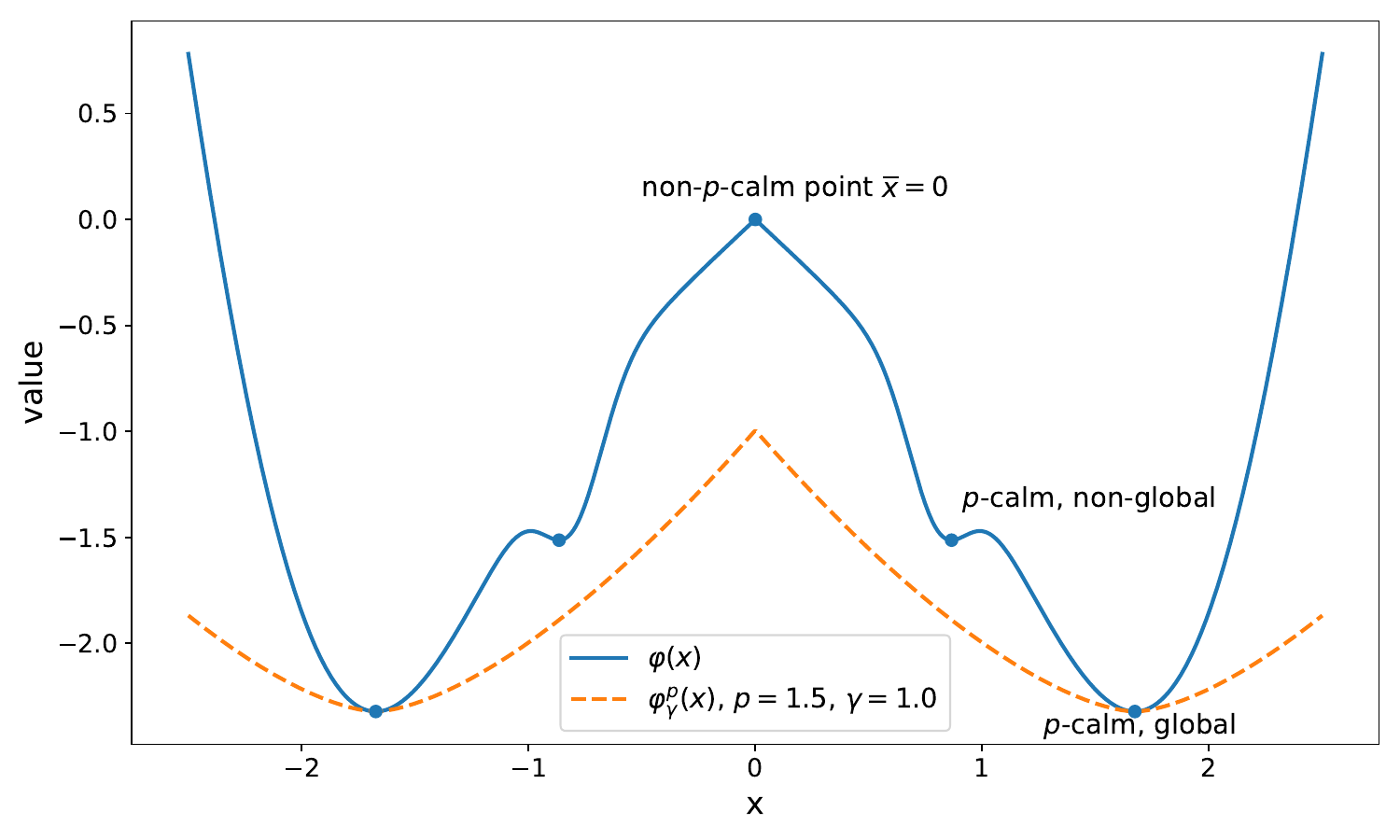}
\caption{Graph of the function $\gf$ and its envelope $\fgam{\gf}{p}{\gamma}$ for $p=\frac{3}{2}$ and $\gamma=1$.
The figure highlights the non-$p$-calm point $\ov{x}=0$, the $p$-calm non-global local minimizer
$x_{\mathrm{loc}}$, and the $p$-calm global minimizer $x^\star$. The loss of differentiability of the
envelope at $\ov{x}=0$ is visible in the plot.}
\label{fig:calm-noncalm}
\end{figure}

\begin{remark}[Importance of $p$-calmness]
By Proposition~\ref{lem:progpcalm}~\ref{lem:progpcalm:mincalm}, every global minimizer of $\gf$ is automatically a $p$-calm point, and consequently $p$-calmness is not an additional restriction at optimal solutions. Example~\ref{ex:p-calm-vs-non-p-calm} also shows that $p$-calmness may hold at points that are not global minimizers.
From the viewpoint of the present paper, the importance of $p$-calmness is that, together with prox-regularity, it provides a natural setting in which the HiFBS mapping may become locally
single-valued and the HiFBE may become differentiable. This, in turn, provides a way to develop smoothing techniques based on the HiFBE. We discuss these aspects further in
Remarks~\ref{rem:relHiFBE_grad}~and~\ref{rem:algorithmic_implication}.
\end{remark}

Building on the notion of the high-order majorant \eqref{eq:highOrdmajf}, some {\it high-order forward-backward algorithms} (HiFBA) have been developed in \cite{kabgani2025fb}, employing regularization terms of arbitrary order $p>1$, namely,
 \begin{equation}\label{eq:fbsGeneral:p}
    x^{k+1}\in \argmint{y\in\R^n}\left\{f(x^k)+\langle \nabla f(x^k), y-x^k\rangle+g(y)+\frac{1}{p\gamma}\Vert y - x^k\Vert^p\right\}.
\end{equation}
This offers improved flexibility and adaptability when the objective function $f$ lacks the $L$-smoothness property. 
We conclude this section with the following observation, which clarifies the relationship between the scaled gradient method on HiFBE and the iterative scheme~\eqref{eq:fbsGeneral:p}.
\begin{remark}[Relation between HiFBA and scaled gradient method]\label{rem:relHiFBE_grad}
Let $\fgam{\gf}{p}{\gamma}\in \mathcal{C}^1(U)$ on an open set $U\subseteq\R^n$,  and consider the sequence $\{x^k\}_{k\in \Nz}\subseteq U$ generated by
the iterative method \eqref{eq:fbsGeneral:p}. For each $k\in \Nz$, we have $x^{k+1}= \Tprox{\gf}{\gamma}{p}(x^k)$, and the gradient of the HiFBE satisfies
\[
\nabla \fgam{\gf}{p}{\gamma}(x^k)=\nabla^2f(x^k)(x^{k+1}-x^k)+\frac{1}{\gamma}\Vert x^k-x^{k+1}\Vert^{p-2}(x^k-x^{k+1}).
\]
Thus, the iterative method \eqref{eq:fbsGeneral:p} can be interpreted as
\begin{equation}\label{eq1:Interpretation}
x^{k+1} = x^k + d^k,
\end{equation}
where the step direction $d^k$ satisfies
\begin{equation}\label{eq2:Interpretation0}
\left(\frac{1}{\gamma}\Vert d^k\Vert^{p-2}I - \nabla^2f(x^k)\right)d^k=-\nabla \fgam{\gf}{p}{\gamma}(x^k),
\end{equation}
which is a {\it scaled gradient method} on HiFBE. For the special case $p=2$, the iteration \eqref{eq1:Interpretation} becomes
\[
x^{k+1} = x^k- \gamma\left(I - \gamma\nabla^2f(x^k)\right)^{-1}\nabla \fgam{\gf}{p}{\gamma}(x^k),
\]
which corresponds to a {\it scaled gradient method} applied to the classical forward–backward envelope, as also discussed in~\cite[Eq.~(2.12)]{Stella17}. 
Moreover, if $f=0$,  then \eqref{eq1:Interpretation} reduces to
\[
x^{k+1} =x^k  - \gamma^{\frac{1}{p-1}}\left\Vert \nabla\fgam{\gf}{p}{\gamma}(x^k)\right\Vert^{\frac{2-p}{p-1}}\nabla\fgam{\gf}{p}{\gamma}(x^k),
\]
which represents a gradient-type method on HOME; see \cite[Remark~39]{Kabganidiff}.
\end{remark}

\begin{remark}[Algorithmic implication of Theorem~\ref{th:weaksm}]\label{rem:algorithmic_implication}
Let the assumptions of Theorem~\ref{th:weaksm} hold, and let $\ov\gamma>0$ denote the threshold provided by Theorems~\ref{th:dif:cont}~and~\ref{th:weaksm}. Then, for every $\gamma\in(0,\ov\gamma)$, there exists a neighborhood $\mb(0;r)$ such that $\Tprox{\gf}{\gamma}{p}$ is single-valued and continuous, and $\fgam{\gf}{p}{\gamma}$ is of class 
$\mathcal{C}^{1,\eta}(\mb(0; r))$  with $\eta=\bs\min\{\mu,\nu^2/2\}$.
Hence, on $\mb(0; r)$, the iteration
\[
x^{k+1}\in \Tprox{\gf}{\gamma}{p}(x^k), 
\]
can be interpreted, in view of Remark~\ref{rem:relHiFBE_grad}, as a scaled gradient-type method applied to the weakly smooth merit function $\fgam{\gf}{p}{\gamma}$.
More precisely, if $\{x^k\}_{k\in\Nz}\subseteq \mb(0;r)$ is generated by \eqref{eq:fbsGeneral:p}, then $x^{k+1}=\Tprox{\gf}{\gamma}{p}(x^k)$ and
\[
\left(\frac{1}{\gamma}\|d^k\|^{p-2}I-\nabla^2 f(x^k)\right)d^k=-\nabla\fgam{\gf}{p}{\gamma}(x^k),
\qquad d^k:=x^{k+1}-x^k.
\]
Therefore, Theorem~\ref{th:weaksm}~\ref{th:weaksm:b} provides exactly the regularity needed to analyze \eqref{eq:fbsGeneral:p}, equivalently \eqref{eq1:Interpretation}-\eqref{eq2:Interpretation0}, by means of the standard descent machinery for gradient-type methods with H\"older continuous gradients.

In particular, once this iteration is supplemented with a globalization mechanism ensuring sufficient decrease of $\fgam{\gf}{p}{\gamma}$ (for instance, a backtracking or line-search rule), one may use
the continuity of $\Tprox{\gf}{\gamma}{p}$ to conclude that every cluster point $x^\star$ satisfying $x^{k+1}-x^k\to 0$ is a proximal fixed point, namely $x^\star=\Tprox{\gf}{\gamma}{p}(x^\star)$.
Equivalently, such a limit point is a stationary point of the envelope dynamics induced by \eqref{eq:fbsGeneral:p}. For convergence of the whole sequence, and in particular for convergence rates, one needs additional assumptions beyond Theorem~\ref{th:weaksm}~\ref{th:weaksm:b}, such as a sufficient-decrease condition and, for rate statements, a Kurdyka-\L{}ojasiewicz or related error-bound hypothesis.

Finally, we emphasize that the admissible parameter $\gamma$ in the above discussion is the same local range as in Theorems~\ref{th:dif:cont}~and~\ref{th:weaksm}, that is, $\gamma$ must be chosen sufficiently small so that the local single-valuedness and continuity of $\Tprox{\gf}{\gamma}{p}$ hold. In particular, according to the proof of Theorem~\ref{th:dif:cont}, $\gamma$ is required to lie below the threshold imposed by the $p$-calmness constant, the H\"older constant of $\nabla f$, the prox-boundedness threshold $\gamma^{g, p}$, and the additional local restriction ensuring uniqueness.
\end{remark}


\section{Conclusion}\label{sec:disc}
This paper provided a rigorous analysis of the high-order forward–backward splitting (HiFBS) mapping and its associated high-order forward–backward envelope (HiFBE) for nonconvex composite optimization. 
We established fundamental regularity properties of these mappings, including boundedness and local uniform boundedness of HiFBS, as well as H\"{o}lder continuity of HiFBE and, under the corresponding additional conditions, its Lipschitz continuity. 
Moreover, we derived explicit expressions for the Fr\'{e}chet and Mordukhovich subdifferentials of HiFBE and obtained necessary and sufficient conditions for differentiability, together with sufficient conditions for weak smoothness.  By exploiting the prox-regularity of $g$ and the notion of $p$-calmness, we further demonstrated local single-valuedness and continuity of HiFBS, which in turn guarantee differentiability of HiFBE near calm points. 
These results advance the theoretical understanding of high-order envelope mappings and unify several perspectives in variational analysis, weak smoothness, and high-order proximal theory.  They provide the analytical foundation required to investigate the convergence and stability of gradient-based first-order methods for composite optimization problems in the nonconvex setting.


\subsection*{\textbf{Funding Information}}
The Research Foundation Flanders (FWO) research project G081222N and UA BOF DocPRO4 projects with ID 46929 and 48996 partially supported the paper's authors.

\bibliographystyle{spbasic}
\bibliography{references}

\end{document}